\documentclass[12pt, a4paper]{article}
\usepackage{euscript, amsmath, amsthm, amssymb,amscd}
\usepackage{cite}
\usepackage{float}
\usepackage{graphicx}
\usepackage{setspace}
\usepackage[all,cmtip]{xy}
\usepackage[english]{babel}
\usepackage{appendix}
\usepackage{MnSymbol}
\usepackage{hyperref}
\let\euf\EuScript 

\newcommand{\proj}{\mathbb{P}}

\newcommand{\seq}{\subseteq}
\newcommand{\C}{\mathbb{C}}

\newcommand{\rank}{\text{rank }}
\newtheorem{thm}{Theorem}[section]
\newtheorem*{thm-nl}{Theorem}
\newtheorem*{prop-nl}{Proposition}

\newtheorem{lem}[thm]{Lemma}

\newtheorem{cor}[thm]{Corollary}
\newtheorem*{cor-nl}{Corollary}
\newtheorem{conjecture}[thm]{Conjecture}
\newtheorem*{conjecture-nl}{Conjecture}
\newtheorem*{quest-nl}{Question}
\newtheorem*{quests-nl}{Questions}

\newtheorem{prop}[thm]{Proposition}

\theoremstyle{remark}
\newtheorem*{rem}{Remark}

\newtheorem{remark}[thm]{Remark}

\title{{The Moduli of Singular Curves on K3 Surfaces}}
\author{Michael Kemeny}

\begin{document}
\maketitle
\begin{abstract}
In this article we consider moduli properties of singular curves on K3 surfaces. Let $\mathcal{B}_g$ denote the stack of primitively polarized K3 surfaces $(X,L)$ of genus $g$ and let $\mathcal{T}^n_{g,k} \to \mathcal{B}_g$ be the stack parametrizing tuples $[(f: C \to X, L)]$ with $f$ an unramified morphism which is birational onto its image, $C$ a smooth curve of genus $p(g,k)-n$ and $f_*C \in |kL|$. We show that the forgetful morphism $$\eta \; : \; \mathcal{T}^n_{g,k}  \to \mathcal{M}_{p(g,k)-n}$$ is generically finite on at least one component, for all but finitely many values of $p(g,k)-n$. We further study the Brill--Noether theory of those curves parametrized by the image of $\eta$, and find a Wahl-type obstruction for a smooth curve with an unordered marking to have a nodal model on a K3 surface in such a way that the marking is the divisor over the nodes.
\end{abstract}
\section{Introduction}
The aim of this article is to study the moduli of those singular curves which may be embedded into a K3 surface. 
Let $[C] \in \mathcal{M}_p$ be a point of the moduli space of smooth curves of genus $p$. We say $C$ admits a singular model lying on a K3 surface of genus $g$ if there exists a polarized K3 surface $(X,L)$ of genus $g$ and an integral curve $D \in |kL|$ for some $k$ such that $C$ is isomorphic to the normalization of $D$.  
Broadly speaking, we wish to consider the following question:
\begin{quests-nl}
What is the dimension of the locus of curves $[C] \in \mathcal{M}_p$ admiting a singular model lying on a K3 surface of genus $g$? Furthermore, what conditions must a curve $[C] \in \mathcal{M}_p$ satisfy in order to admit a singular model $D$ lying on a K3 surface?
\end{quests-nl}

In practice, one tends to put a condition on the singularities of the integral curve $D$ in order to approach the above question, as otherwise the deformation theory of the pair $(D,X)$ is hard to control. In \cite{flam}, the above questions are considered under the hypothesis that $D$ is nodal. We will instead work with the much weaker hypothesis that the normalization morphism $f: C \to D$ is unramified (recall that $f$ is said to be unramified if its differential never vanishes). In particular, if $D$ has ordinary singularities then $f$ is unramified, whereas $f$ has ramification if $D$ has a cusp.

If $D$ is a singular, integral curve on a K3 surface $X$, let $\mu: C:=\tilde{D} \to X$ denote the composition of the normalization $\tilde{D} \to D$ with the inclusion $D \hookrightarrow X$, and let $p$ be the arithmetic genus of $C:=\tilde{D}$. This gives a one-to-one correspondence between pairs $(D,X)$, where $D$ is integral of geometric genus $p$, and morphisms $f: C \to X$ where $C$ is a smooth curve of genus $p$ and $f$ is birational onto its image. As is by now well-known, the deformation theory of the morphism $f$ is in many ways considerably easier to work with than the deformation theory of the pair $(D,X)$. We will take this viewpoint throughout this paper and formulate the above questions in terms of stable maps, see \cite{fulpar} or \cite{arakol} for an excellent introduction to this topic.

\subsection{The number of moduli of singular curves on K3 surfaces}
Let $\mathcal{B}_g$ denote the stack of pairs $(X,L)$, where $X$ is a K3 surface over $\C$ and $L$ is an ample, primitive polarization with $(L)^2=2g-2$ for $g \geq 3$. There is a Deligne--Mumford stack $$\mathcal{W}^{n}_{g,k} \to \mathcal{B}_g$$ with fibre over a polarized K3 surface $[(X,L)] \in \mathcal{B}_g$ parametrizing all stable maps
$f: C \to X$ with $f_*C \in |kL|$, where $C$ is a connected, nodal curve of arithmetic genus $p(g,k)-n$, with $p(g,k):=k^2(g-1)+1$. Denote by $$\mathcal{T}^n_{g,k} \seq \mathcal{W}^{n}_{g,k}$$ the open subset consisting of unramified stable maps $f: C \to X$ with $C$ integral and smooth such that $f$ is birational onto its image. 

By the (reduced) deformation theory of stable maps, $\mathcal{T}^n_{g,k}$ is a smooth stack of dimension $p(g,k)-n+19$, and every component of $\mathcal{T}^n_{g,k}$ dominates $\mathcal{B}_g$, see \cite[\S 4]{kemeny-thesis}, \cite[\S 2]{huy-kem} and Section \ref{finny} of this work. For the general deformation theory of a morphism see \cite{flenner-ueber}, \cite[\S 7.4]{buchweitz-flenner}. Also see \cite{kool-thomas}, in particular Theorem 2.4 and Remark 3.1, for a different approach. The stack $\mathcal{T}^n_{g,k}$ is nonempty by a result of Chen, \cite{chen-rational}. For a different and more detailed account of Chen's theorem, see \cite{galati-knutsen}.

For $p(g,k)-n \geq 2$, there is a forgetful morphism
$$
\eta \; : \; \mathcal{T}^n_{g,k}  \to \mathcal{M}_{p(g,k)-n} $$
defined by taking $[(f: C \to X, L)] \in \mathcal{T}^n_{g,k} $ to  $ [C]$, where $\mathcal{M}_{p(g,k)-n}$ denotes the stack of smooth curves of genus $p(g,k)-n$.  A dimension count suggests that this might be dominant for $2 \leq p(g,k)-n \leq 11$ and generically finite for $p(g,k)-n \geq 11$. 
To ease the notation in the primitive case $k=1$ write $\mathcal{T}^n_{g}:=\mathcal{T}^n_{g,1} $.

The case $n=0$ has been studied in depth. It was shown in \cite{mori-mukai} and \cite[Thm.\ 7]{mukai-fano} that if $n=0$, $k=1$ then the morphism 
$$\eta \; : \; \mathcal{T}^0_{g}  \to \mathcal{M}_{g}$$ is generically finite for $g \geq 13$ or $g=11$. In the non-primitive case $k \geq 2$, a very different approach using the deformation theory of cones shows that $\eta$ is generically finite for $g \geq 7$ and $n=0$, \cite{cili-classification}. 

Our first result is an extension of the results on generic finiteness to the singular case $n >0$. In the case $k=1$, we show:
\begin{thm} \label{finiteness}
Assume $g \geq 11$, $n \geq 0$. Set $m=\left \lfloor \frac{g-11}{6} \right \rfloor$ and let $0 \leq r(g) \leq 5$ be the unique integer such that 
$$g-11 = 6m +r(g).$$
Define 
\begin{itemize}  
\item $l_g:=12$, if $r(g)=0$.
\item $l_g:=13$, if $1 \leq r(g) <5$.
\item $l_g:=15$ if $r(g)=5$.
\end{itemize}
Then there is a component $I \seq \mathcal{T}^n_{g}$ such that $${\eta}_{|_I}: I \to \mathcal{M}_{g-n}$$ is generically finite for
$g-n \geq l_g$. For the general $[f: C\to X] \in I$, $C$ is non-trigonal.
\end{thm}
In the case $k \geq 2$ we show:
\begin{thm} \label{finiteness-nonprim}
Assume $k \geq 2$, $g \geq 8$. Set $m:= \left \lfloor \frac{g-5}{6} \right \rfloor$ and let $0 \leq r(g) \leq 5$ be the unique integer such that
$$ g-5=6m+r(g).$$
Define:
\begin{itemize}  
\item $l_g:=15$, if $r(g)=3,4$, $m$ odd and/or $k$ even.
\item $l_g:=16$, if $r(g)=3,4$, $m$ even and $k$ odd.
\item $l_g:=17$, if $r(g)=5$, $m$ odd and/or $k$ even.
\item $l_g:=18$, if $r(g)=5$, $m$ even and $k$ odd.
\item $l_g:=17$, if $r(g)\leq 2$, $m$ even and/or $k$ even.
\item $l_g:=18$, if $r(g) \leq 2$, $m$ odd and $k$ odd.
\end{itemize}
Then there is a component $I \seq \mathcal{T}^n_{g,k}$ such that $${\eta}_{|_I}: I \to \mathcal{M}_{p(g,k)-n}$$ is generically finite for
$p(g,k)-n \geq l_g$. For the general $[f: C \to X] \in I$, $C$ is non-trigonal.
\end{thm}
Setting $n=0$, we recover the (optimal) statement in the smooth, primitive, case and all cases other than $g=7$ in the nonprimitive case.\footnote{The well-known case $g=11$, $n=0$, $k=1$ does not follow from Theorem \ref{finiteness} but rather from Corollary \ref{bij-cor}.} In particular, this gives a new proof of the generic finiteness theorem for $n=0$, $k \geq 2$, $g \geq 8$ which resembles the original approach of \cite{mori-mukai}.

A key part of our proof was inspired from a close reading of Mukai's papers \cite{mukai-nonabelian}, \cite{mukai-duality}, \cite{mukai-genus11}. In particular, we use Mukai's proof that, in some special cases, the data of a closed embedding of a smooth, genus $11$ curve $C$ in a K3 surface $X$ can be described purely in terms of $C$ together with a special line bundle $A \in W^1_6(C)$. To be more precise, the K3 surface $X$ is in these special cases isomorphic to the quadric hull of $\phi_{A^{\dagger}}: C \hookrightarrow \proj^k$, where $A^{\dagger}$ is the adjoint of $A$. This isomorphism identifies the given closed embedding $C \hookrightarrow S$ with the natural embedding of $C$ into the quadric hull of $\phi_{A^{\dagger}}$ (up to automorphisms of $\proj^k$). Using results of \cite{donagi-morrison}, one can further check that the bundle $A \in W^1_6(C)$ is unique. In particular, arguments along these lines establish the main result of \cite{mori-mukai} in the case $g=11$. See Section \ref{mukai} for more details.

Our new idea is to show that Mukai's argument applies to the particular K3 surfaces $Y_{\Omega_{11}}$ introduced in Lemma \ref{lem-aaa}. These K3 surfaces, which, as far as we know, have not previously been considered, are closely related to the Picard rank two K3 surfaces from \cite[\S 3]{mukai-duality}. In contrast, however, to the surfaces studied by Mukai, $Y_{\Omega_{11}}$ has Picard rank ten and contains multiple disjoint smooth rational curves by construction. As a result, some involved lattice calculations must be carried out in Section \ref{mukai}. Using that rational curves are rigid on K3 surfaces, one then uses the surfaces $Y_{\Omega_{11}}$ to build Theorems \ref{finiteness} and \ref{finiteness-nonprim} up from the main result in \cite{mori-mukai}. This is performed in Section \ref{finny} via the deformation theory of stable maps.

Denote by 
$$\mathcal{V}^{n}_{g,k} \seq \mathcal{T}^n_{g,k}$$ the open substack parametrizing morphisms $[(f: C \to X,L)]$ with $f(C)$ nodal and set $\mathcal{V}^{n}_{g}=\mathcal{V}^{n}_{g,1}$.
The following conjecture is found in \cite{dedieu}:
\begin{conjecture}
The moduli space $\mathcal{V}^n_{g,k}$ is irreducible.
\end{conjecture}
Let $I \seq \mathcal{T}^n_{g,k}$ be an irreducible component and denote the fibre of $I \to \mathcal{B}_g$ over $[(X,L)]$ by $I(X,L)$. If $(X,L)$ is general, then each component of $I(X,L)$ has dimension $p(g,k)-n$ and we have an injective morphism $I(X,L) \to |kL|$ sending $[f: C \to X]$ to the integral curve $[f(C)]$. Assume $p(g,k)-n>0$. If $I(X,L)$ contains a map $[(f: C \to X, L)] $ with $C$ non-trigonal\footnote{The realisation that the arguments need a non-trigonality assumption, erroneously omitted in \cite[Lemma 3.1]{chen-rational}, is due to Dedieu--Sernesi, \cite{ded-sern}. Also compare with \cite[Lemma 3.43]{harris-morrison}.} then a result of Harris and Chen, with an error corrected by Dedieu--Sernesi, shows that the component $I(X,L)$ must meet $ \mathcal{V}^n_{g,k}$, see \cite[Pg.\ 107ff]{harris-morrison}, \cite{harrissev}, \cite[Lemma 3.1]{chen-rational}, \cite[Thm.\ 2.8]{ded-sern}. Thus we have:
\begin{cor} \label{univ-sev-corollary-sdfw}
The restriction $${\eta}_{|_{\mathcal{V}^{n}_{g,k}}}: \mathcal{V}^{n}_{g,k} \to \mathcal{M}_{p(g,k)-n}$$ is generically finite on one component, for the same bounds on $p(g,k)-n$ as in Theorem \ref{finiteness} and \ref{finiteness-nonprim}.
\end{cor}
We should also mention here that C.\ \!Ciliberto, F.\ \!Flamini, C.\ \!Galati and A.\ \!Knutsen have a very different approach using degenerations of sheaves to unions of rational scrolls, which when completed is likely to give another proof of Corollary \ref{univ-sev-corollary-sdfw} (for certain bounds on  $p(g,k)-n$). Moreover, their approach may potentially prove the (local) surjectivity of $\eta$ on one component of $\mathcal{V}^{n}_{g,k}$ for some cases within the range $p(g,k)-n \leq 11$, which is beyond the reach of our method.

\subsection{An obstruction for a marked curve to admit a nodal model on a K3 surface}
It is a natural question to study the image of $\eta$.  In the case of smooth curves $n=0$, there is a well-known conjectural characterization of the image $\eta$, due to Wahl \cite{wahl-square}. He makes the following remarkable conjecture, which would give a complete characterization of those smooth curves that lie on a K3 surface:
\begin{conjecture}[Wahl] \label{wahlconj}
Assume $C$ is a smooth curve of genus $g \geq 8$ which is Brill--Noether general. Then there exists a K3 surface $X \seq \proj^g$ such that $C$ is a hyperplane section of $X$ if and only if the Wahl map $W_C$ is nonsurjective. 
\end{conjecture}
Here the Wahl map refers to the map $\bigwedge^2 H^0(C,K_C) \to H^0(C,K_C^3)$ given by 
$s \wedge t \mapsto tds-sdt$. One side of this conjecture is well-known; indeed if $C \seq X$ is a smooth curve in a K3 surface then $W_C$ is nonsurjective, \cite{wahl-jac}. Furthermore, if $C$ is general and $Pic(X) \simeq \mathbb{Z}C$, then $C$ is Brill--Noether general, \cite{lazarsfeld-bnp}. In \cite[Question 5.5]{flam}, it was asked if there exists such a Wahl-type obstruction for a smooth curve to have a nodal model lying on a K3 surface.\footnote{An obstruction was also proposed in \cite{halic}, the proof however seems flawed, see Remark \ref{halic-remarks}.} 

Let $\widetilde{\mathcal{M}}_{p(g,k)-n,2n} := \overline{\mathcal{M}}_{p(g,k)-n,2n}/ S_{2n}$ denote the stack of curves with an unordered marking (or divisor). One may slightly alter the above question and ask if there exists an obstruction for a \emph{marked} curve to have a nodal model lying on a K3 surface in such a way that the marking is the divisor over the nodes (when we forget about the ordering).
For  any positive integers $h,l$ and $[(C,T)] \in \widetilde{\mathcal{M}}_{h,2l}$, one may consider the Gaussian
$$W_{C,T}: \bigwedge^2 H^0(C,K_{C}(-T)) \to H^0(C,K_{C}^3(-2T)) $$
which we will call the \emph{marked Wahl map}, since it depends on both the curve and the marking. In Section \ref{markywahl} we use a method inspired from \cite{wahl-plane-nodal} to show the following:
\begin{thm} \label{inf-many-wahl}
Fix any integer $l \in \mathbb{Z}$. Then there exist infinitely many integers $h(l)$, such that the general marked curve $[(C,T)] \in \widetilde{\mathcal{M}}_{h(l),2l}$ has surjective marked Wahl map.
\end{thm}
On the other hand we show:
\begin{thm} \label{marked-wahl-k3}
Assume $g-n \geq 13$ for $k=1$ or $g \geq 8$ for $k >1$, and let $ n \leq \frac{p(g,k)-2}{5}$. Then there is an irreducible component $I^0 \seq \mathcal{V}^n_{g,k}$ such that for a general $[(f: C \to X,L)] \in I^0$ the marked Wahl map $W_{C,T}$ is nonsurjective, where $T \seq C$ is the divisor over the nodes of $f(C)$.
\end{thm}

\subsection{Brill--Noether theory for nodal curves on K3 surfaces}
In the last section we study the Brill--Noether theory of nodal curves on K3 surfaces. There are two related questions: for $[(f: C \to X,L)] \in \mathcal{V}_g^n$ general, one may firstly ask if the smooth curve $C$ is Brill--Noether general and secondly if the nodal curve $f(C)$ is Brill--Noether general. For the first question we show in Section \ref{BNP-nodal}:
\begin{prop} \label{BNP-theorem}
Assume $g-n \geq 8$. Then there exists a component $\mathcal{J} \seq \mathcal{V}^n_g$ such that for $[(f:C \to X,L)] \in \mathcal{J}$ general, $C$ is Brill--Noether--Petri general.
\end{prop}
The above result should not be expected to hold for all $[(f:C \to X,L)] \in \mathcal{J}$ (or even for all $[(f:C \to X,L)] \in \mathcal{J}$ with the general polarized K3 surface $(X,L)$ kept fixed), see \cite[Thm.\ 0.1]{ciliberto-knutsen-gonal}.

For the second question we again have a positive answer. For an integral nodal curve $D$, we denote by $\bar{J}^d(D)$ the compactified Jacobian of degree $d$, rank one, torsion-free sheaves on $D$. In Section \ref{rat} we use moduli spaces of sheaves as in \cite{ogrady} to show:
\begin{thm} \label{bn-rat-thm}
Let $X$ be a projective K3 surface with $Pic(X) \simeq \mathbb{Z} L$ and $(L \cdot L)=2g-2$. Suppose $D \in |L|$ is a rational, nodal curve. Then
$$\overline{W}^r_d(D) := \{ \text{$A \in \bar{J}^d(D)$ with $h^0(A) \geq r+1$} \}$$ is either empty or is equidimensional of the expected dimension $\rho(g,r,d)$.
\end{thm} 
As one may smoothen the nodes of a rational nodal curve $D$ on a K3 surface to produce a curve with an arbitrary number of nodes, the above result immediately gives the following corollary:
\begin{cor}
For any $n \geq 0$, there is a component $\mathcal{J} \seq \mathcal{V}^n_g$ such that if $[(f:C \to X,L)] \in \mathcal{J}$ is general and $D=f(C)$ then
$$\overline{W}^r_d(D) := \{ \text{$A \in \bar{J}^d(D)$ with $h^0(A) \geq r+1$} \}$$ is either empty or is equidimensional of the expected dimension $\rho(g,r,d)$. 
\end{cor}
In particular, if $\rho(g,r,d) <0$, $\overline{W}^r_d(D) = \emptyset$ for $D$ as in the above corollary; indeed this is well-known and follows from the arguments of \cite[\S 3.2]{gomez}, \cite[Cor.\ 1.4]{lazarsfeld-bnp}. On the other hand, if $\rho(g,r,d)  \geq 0$, then $\overline{W}^r_d(D) \neq \emptyset$ by deforming $D$ to a smooth curve on $X$ and semicontinuity.

We may summarize the above results as stating that there are no Brill--Noether obstructions for a curve to have a nodal model lying on a K3 surface. It would be interesting to find \emph{non-abelian}, rank two, Brill--Noether obstructions for a curve to have a nodal model lying on a K3 surface, in analogy with the smooth case, \cite{mukai-nonabelian}, \cite{sernesi-brill}.

\medskip {\bf Acknowledgments:} I am extremely thankful to my thesis advisor, Professor Daniel Huybrechts, for introducing me to this topic and for patiently reading multiple versions of this article. I would also like to thank the participants of the ``Workshop on Severi Varieties and Hyperk\"ahler Manifolds", November 4-8, 2013 in Rome, where some of these results were announced, for useful discussions. In particular, the author benefitted greatly from discussions with C.\ \!Ciliberto, F.\ \!Flamini, C.\ \!Galati, A.\ \!Knutsen and E.\ \!Sernesi. I am further thankful to A.\ \!Knutsen for helpful discussions and for pointing out an error in an earlier draft of $\S 4$, and would like to thank my colleagues U.\ \!Greiner and S.\ \!Schreieder for useful discussions related to $\S 2$, $\S 3$. Last but not least, I most heartily thank the referee for a very thorough reading and for taking the time to compile a long list of misprints. In particular, I am thankful to the referee for correcting an error in the bound of Proposition \ref{mark-surj-main}. This work was funded via a PhD Scholarship from the Bonn International Graduate School in Mathematics (BIGS) and by SFB/TR 45.

\section{Mukai's theory for curves on K3 surfaces} \label{mukai}
In this section we will recall a construction of Mukai to construct loci $Z \seq \mathcal{V}^0_g$ such that  for $x \in Z$ the fibre of  
$\eta : \mathcal{V}^0_g \to \mathcal{M}_{g}$ over $\eta(x)$ is zero-dimensional at $x$. The main point is that, in some special cases, the data of a closed embedding of a smooth, genus $11$ curve $C$ in a K3 surface $X$ can be described purely in terms of $C$ together with a special line bundle $A \in W^1_6(C)$. This will be our basic tool for studying the generic finiteness of the morphism $\eta: \mathcal{T}^n_{g,k} \to \mathcal{M}_{p(g,k)-n}$. The main result in this section is Corollary \ref{bij-cor}.

Let $g\geq 5$ be an integer, let $1 \leq d_1,d_2, \ldots, d_8 < \left \lfloor \frac{g+1}{2} \right \rfloor$ be integers, and consider first the rank ten lattice $\Omega_g$ with ordered basis $\{L, E, \Gamma_1, \ldots, \Gamma_8 \}$ and with intersection form given by:
\begin{itemize}
\item $(L \cdot L)=2g-2$
\item $(L \cdot E)=  \left \lfloor \frac{g+1}{2} \right \rfloor$
\item $(E \cdot E)=0$
\item $(\Gamma_i)^2=-2$ for $1 \leq i \leq 8$
\item $(E \cdot \Gamma_i)=0$ for $1 \leq i \leq 8$
\item $(L \cdot \Gamma_i)=d_i$ for $1 \leq i \leq 8$
\item $(\Gamma_i \cdot \Gamma_j)=0$ for $i \neq j$, $1 \leq i,j \leq 8$
\end{itemize}
It is easily seen that the above lattice has signature $(1,9)$ and is even.
\begin{lem} \label{lem-aaa}
 Let $g \geq 6$ be an integer and choose $1 \leq d_1, \ldots, d_8 < \left \lfloor \frac{g+1}{2} \right \rfloor$. There exists a K3 surface $Y_{\Omega_g}$ with $Pic(Y_{\Omega_g}) \simeq \Omega_g$. Furthermore, for any such K3 we may choose the ordered basis $\{L,E, \Gamma_1, \ldots, \Gamma_8 \}$ of $\Omega_g$ 
 in such a way that $L-E$ is big and nef and with $\Gamma_i$ and $E$ representable by smooth, integral curves for $1 \leq i \leq 8$. Further there is a smooth rational curve $\widetilde{\Gamma}_i \in |E-\Gamma_i|$.
\end{lem}
\begin{proof}
By the global Torelli theorem and from a result of Nikulin, the fact that this lattice has signature $(1,9)$ and is even implies that there exists a K3 surface $Y_{\Omega_g}$ with $Pic(Y_{\Omega_g}) \simeq \Omega_g$, \cite[Cor.\ 1.9, Cor.\ 2.9]{morrison-large} or \cite[Cor.\ 14.3.1]{huy-lec-k3}. By performing Picard--Lefschetz reflections and a sign change, we may assume that $L-E$ is big and nef, since $(L-E \cdot L-E)>0$, \cite[Prop.\ VIII.3.9]{barth} and \cite[Cor.\ 8.2.11]{huy-lec-k3}. Next, $(E \cdot E)^2=0$ and $(L-E \cdot E) =\left \lfloor \frac{g+1}{2} \right \rfloor>0$ which implies that $E$ is effective, \cite[Prop.\ VIII.3.6(i)]{barth}. We now want to show that the general element of $|E|$ is smooth and irreducible. By \cite[Prop.\ 2.6]{donat} and the fact that $E$ belongs to a basis of $Pic(Y_{\Omega_g})$ it is enough to show that $|E|$ is base-point free. It is enough to show that $E$ is nef, \cite[Lemma 2.3]{knut} or \cite[Prop.\ 2.3.10]{huy-lec-k3}. 

So it suffices to show there is no effective divisor $R$ with $(R)^2=-2$ and $(E \cdot R)<0$, \cite[Prop.\ VIII.3.6]{barth}. Suppose for a contradiction that such an $R$ exists. Write $R=xL+y E+ \sum_{i=1}^{8} z_{i} \Gamma_i$ for integers $x,y,z_{i}$.  As $(E \cdot R) = x \left \lfloor \frac{g+1}{2} \right \rfloor <0$ we have $x <0$. Then $(R-x L)^2=-2 \sum_{i=1}^8 z^2_{i}  \leq 0$. However, 
\begin{align*}
(R-x L)^2 &= -2+x^2(2g-2)-2x(L \cdot R) \\
&=-2+x^2(2g-2-2\left \lfloor \frac{g+1}{2} \right \rfloor)-2x(L-E \cdot R) \\
& >0
\end{align*}
for $x <0$ and $g \geq 6$.
Thus $|E|$ is an elliptic pencil. 

Next $\Gamma_1$ is effective, since $(\Gamma_1 \cdot L-E) >0$. We claim $\Gamma_1$ is integral. Otherwise, there would be an integral component $R$ of $\Gamma_1$ with $(R \cdot \Gamma_1) <0$, since $(\Gamma_1)^2=-2$. Further, $(R)^2=-2$, since $R$ is not nef. Write $R=xL+y E+\sum_{i=1}^{8} z_{i} \Gamma_i$. We have $(R \cdot E)=x \left \lfloor \frac{g+1}{2} \right \rfloor \geq 0$ so $x \geq 0$. Assume $x \neq 0$. Then we have $(R \cdot R+E) >0$ and $(R+E)^2>0$ so $R+E$ is big and nef, which contradicts that $(R+E \cdot \Gamma_1)=(R \cdot \Gamma_1)<0$. So $x=0$. But then $(R)^2=-2$ gives $\sum_{i=1}^8 z^2_{i}=1$, and $(R \cdot \Gamma_1)=-2 z_{1} <0$ so $z_{1}=1$ and $z_{i}=0$, $i>1$. Lastly, $(R \cdot L)=d_1+y \left \lfloor \frac{g+1}{2} \right \rfloor \geq 0$ so we must have $y \geq 0$ (as  $d_1< \left \lfloor \frac{g+1}{2} \right \rfloor$). Since $R$ is a smooth and irreducible rational curve, we must then have $y=0$ and $R=\Gamma_1$ (as the only effective divisor in $|R|$ is integral). Thus $\Gamma_1$ is integral. Likewise, $\Gamma_2, \ldots, \Gamma_8$ are integral.

Next, $\widetilde{\Gamma}_1$ is effective, since $(\widetilde{\Gamma}_1)^2=-2$ and $(\widetilde{\Gamma}_1 \cdot L-E) >0$. 
Let $R$ be an integral component of $\widetilde{\Gamma}_1$ such that $(R \cdot \widetilde{\Gamma}_1) <0$, $(R)^2=-2$. 
Writing $R=xL+y E+\sum_{i=1}^{8} z_{i} \Gamma_i$, we see as above $x=0$ and we must have $z_1=-1$ and $z_{i}=0$, $i>1$. Since $( R \cdot L)=-d_1+y \left \lfloor \frac{g+1}{2} \right \rfloor \geq 0$ we must have $y \geq 1$ and then $R \sim \widetilde{\Gamma}_1+(y-1)E$. Since $R$ is integral, this forces $R= \widetilde{\Gamma}_1$.
\end{proof}
\begin{lem} \label{I2}
Let $Y_{\Omega_g}$ and $\{L,E, \Gamma_1, \ldots, \Gamma_8 \}$ be as in the previous lemma, and let $Y_{\Omega_g} \to \proj^1$ be the fibration induced by $E$. If $Y_{\Omega_g}$ is a general $\Omega_g$-polarized K3 surface, then at least six of the reducible singular fibres $\Gamma_i+\widetilde{\Gamma_i}$ for $1 \leq i \leq 8$ are of the type $I_2$ (as opposed to the type $III$). We choose indices such that $\Gamma_i+\widetilde{\Gamma_i}$
is an $I_2$ fibre for $i \geq 3$.
\end{lem}
\begin{proof}
As $( xL+y E+\sum_{i=1}^{8} z_{i} \Gamma_i \cdot E)=x(L \cdot E)$, the elliptic fibration $Y_{\Omega_g} \to \proj^1$ induced by $E$ has multisection index $(L \cdot E)$, in the sense of \cite[\S 2]{keum}.
Consider the Jacobian fibration $J(Y_{\Omega_g}) \to \proj^1$, which has isomorphic singular fibres to $Y_{\Omega_g}$ (see \cite[Ch.\ 5]{cossec-dolgachev}, \cite[Ch.\ 1.5]{fried} and \cite[\S 11.4]{huy-lec-k3} for the basic theory of Jacobian fibrations). Let $\widetilde{\Omega_g}$ be the lattice generated by $\Omega_g$ and $$F:=E /(L \cdot E).$$ Then $J(Y_{\Omega_g})$ is an element of the global moduli space $M_{\widetilde{\Omega_g}}$ of $\widetilde{\Omega_g}$-polarized K3 surfaces, \cite[Lemma 2.1]{keum} and \cite[Def.\ p.1602]{dolgachev}. A result of Mukai states that one may describe $J(Y_{\Omega_g})$ as a moduli space $M_H(v)$ of $H$-semistable sheaves on $Y_{\Omega_g}$ for a generic polarization $H$ and Mukai vector $(0,E, (L \cdot E))$; see \cite[\S 11.4.2]{huy-lec-k3}, \cite[p.\ 2081]{keum}, \cite{mukai-moduli} (note that for $H$ generic, $M_H(v)$ is irreducible, \cite[Cor.\ 10.3.5]{huy-lec-k3}). This construction can be done in families by \cite[Thm.\ 4.3.7]{huybrechts-sheaves}, so there is a holomorphic map 
\begin{align*}
\phi \; : \; U &\to M_{\widetilde{\Omega_g}} \\
(Y'_{\Omega_g}, \Omega_g \hookrightarrow \text{Pic}(Y'_{\Omega_g})) & \mapsto (J(Y'_{\Omega_g}), \widetilde{\Omega_g} \hookrightarrow \text{Pic}(J(Y'_{\Omega_g})))
\end{align*} 
defined in a Euclidean open subset $U$ about $(Y_{\Omega_g}, \Omega_g \hookrightarrow \text{Pic}(Y_{\Omega_g}))$ in the local moduli space (or period domain) of marked $\Omega_g$-polarised K3 surfaces. The Tate--\v{S}afarevi\v{c} group of  $J(Y_{\Omega_g})$ is countable, \cite[Rem.\ 11.5.12]{huy-lec-k3} (cf.\ \cite[\S 1.5.3]{fried}, \cite[Cor.\ 2.2]{grothendieck-brauer}, \cite[Prop.\ 11.5.6]{huy-lec-k3} and \cite[Thm.\ 0.1]{huybrechts-brauer}). In particular, this implies that the fibres of $\phi$ are zero-dimensional by \cite[Cor.\ 11.5.5]{huy-lec-k3}. By Sard's theorem, $\phi(U)$ contains a Euclidean open set. Thus it suffices to show that there is a dense open set $V \seq M_{\widetilde{\Omega_g}}$ with the property that for any $(X, \widetilde{\Omega_g} \hookrightarrow \text{Pic}(X)) \in V$, the fibration $X \to \proj^1$ induced by $F$ has at least six $I_2$ fibres. 

Let $T$ be the trivial lattice of $J(Y_{\Omega_g}) \to \proj^1$, i.e.\ the lattice generated by $F$, any section $\sigma$ of  $J(Y_{\Omega_g}) \to \proj^1$ and the components of the reducible fibres which do not meet $\sigma$. We have
$$T \simeq \mathfrak{h} \oplus (-2)^8,$$
where $\mathfrak{h}$ denotes the hyperbolic lattice.
By choosing the basis $\{L-gF, F, \Gamma_1-d_1 F, \ldots, \Gamma_8-d_8 F  \}$, we see $\widetilde{\Omega_g}$ is isometric to $T$, or equivalently, the Mordell--Weil group of $J(Y_{\Omega_g})$ is trivial, \cite[Thm.\ 6.3]{schuett}. Suppose $(X, \widetilde{\Omega_g} \hookrightarrow \text{Pic}(X)) \in  M_{\widetilde{\Omega_g}}$ has the property that $X \to \proj^1$ induced by $F$ has at least six $I_2$ fibres. Then the same holds in a dense open set in each component of $M_{\widetilde{\Omega_g}}$ containing $(X, \widetilde{\Omega_g} \hookrightarrow \text{Pic}(X))$ as the condition that a fibre be nodal is Zariski open (i.e.\ for a flat, proper algebraic family of integral curves, the locus of non-nodal curves is a Zariski closed subset of the base). Furthermore, the same clearly holds for the complex conjugate $X^c \to \proj^1$. There are at most two components of $M_{\widetilde{\Omega_g}}$, which locally on the period domain are interchanged by complex conjugation, so this would complete the proof.

 Thus it suffices to find such an elliptic K3 surface $X \to \proj^1$. From \cite[Thm.\ 2.12]{shimada-arxiv} (published in compressed form as \cite{shimada-mich}), there exists an elliptic K3 surface $X \to \proj^1$ with section and torsion-free Mordell--Weil group such that $X$ has 10 singular fibres of type $A_1$, each of which are either of type $I_2$ or $III$. Using that the Euler number of $III$ is $3$, we find that $X$ can have at most $4$ fibres of type $III$, and thus has at least $6$ fibres of type $I_2$, which have Euler number $2$ (as the sum of the Euler numbers of the singular fibres must be $24$, the Euler number of $X$). There is an obvious primitive embedding of $T$ into the trivial lattice of $X$ such that $6$ of the generators correspond to components of the $I_2$ fibres avoiding the section. As the Mordell--Weil group of $X \to \proj^1$ is torsion free, we have a primitive embedding $T \hookrightarrow \text{Pic}(X)$. This completes the proof. 
\end{proof}

\begin{lem} \label{little-lem}
Let $Y_{\Omega_g}$ and $\{L,E, \Gamma_1, \ldots, \Gamma_8 \}$ be as in the previous lemma. If we assume $g > 7$, then $L-E$ is very ample (and hence $L$ is also very ample).
\end{lem}
\begin{proof}
Suppose the big and nef line bundle $L-E$ is not very ample. Then there exists either a smooth rational curve $R \seq Y_{\Omega_g}$ with $(L-E \cdot R) = 0$ or a smooth elliptic curve $F \seq Y_{\Omega_g}$ with $0 < (L-E \cdot F) \leq 2$ (or both exist), by \cite[Thm.\ 1.1]{knut} (set $k=1$ in Knutsen's theorem and note that $L-E$ is primitive). 

Assume firstly that $R$ as above exists; we may write $R=x_1L+y_1 E+ \sum_{i=1}^{8} z_{1,i} \Gamma_i$ for integers $x_1,y_1,z_{1,i}$. We have
\begin{align*}
-2 \sum_{i=1}^8 z^2_{1,i} &=(R-x_1L)^2 \\
&=-2+x_1^2(2g-2)-2x_1(L \cdot R) \\
&=-2+x_1^2(2g-2-2\left \lfloor \frac{g+1}{2} \right \rfloor)-2x_1(L-E \cdot R).
\end{align*}
By the above equality, using that $g \geq 8$ and $(L-E \cdot R)=0$, we find that $x_1=0$ and there exists some $j$ such that $z_{1,j}=\pm 1$, $z_{1,i}=0$ for $i \neq j$. Then $0=( L-E \cdot R)=\pm d_j+y_1 \left \lfloor \frac{g+1}{2} \right \rfloor$ which is impossible
for $1 \leq d_j < \left \lfloor \frac{g+1}{2} \right \rfloor$. 

So now suppose there is some smooth elliptic curve $F \seq Y_{\Omega_g}$ with $0 < (L-E \cdot F) \leq 2$. We may write $F=x_2 L+y_2 E+ \sum_{i=1}^{8} z_{2,i} \Gamma_i$ for integers $x_2,y_2,z_{2,i}$. We have $(F \cdot E)=x_2 (L \cdot E)$ and hence $x_2 >0$ (as $F \notin |E|$).  We calculate
\begin{align*}
-2 \sum_{i=1}^8 z^2_{2,i}  &= (F-x_2L)^2 \\
&= x_2^2(2g-2)-2x_2(L \cdot F) \\
&= x_2(x_2((2g-2)-2 \left \lfloor \frac{g+1}{2} \right \rfloor)-2 (L-E \cdot F)) 
\end{align*}
which is impossible for $g>7$, $0 \leq (L-E \cdot F) \leq 2$. Thus $L-E$ is very ample. Using Knutsen's criterion again, and the fact that $E$ is nef, we see that $L$ is likewise very ample.
\end{proof}

For the rest of the section we will assume $g$ is odd. The following technical lemma will be needed later in this section.
\begin{lem} \label{little-lem2}
Assume $g \geq 11$  is odd and let $Y_{\Omega_g}$ and $\{L,E, \Gamma_1, \ldots, \Gamma_8 \}$  be as in Lemma \ref{little-lem}. Then $L-2E$ is not effective. Further $(L-E)^2 \geq 8$ and there exists no effective divisor $F$ with $(F)^2=0$ and $(F \cdot L-E) \leq 3$.
\end{lem}
\begin{proof}
Suppose $L-2E$ is an effective divisor and let $D_1, \ldots D_k$ be its irreducible components. Write $D_i=x_iL+y_iE+\sum_{j=1}^8 z_{i,j} \Gamma_j$ for integers $x_i,y_i,z_{i,j}$. Then $0 \leq (D_i \cdot E)=x_i (L \cdot E)$ so that $x_i \geq 0$ and $\sum_i x_i=1$. Thus we may assume $x_1=1$ and $x_i=0$ for all $i \geq 2$. Now let $\widetilde{D}$ be any irreducible curve of the form $a E+\sum_{j=1}^8 b_j \Gamma_j$ for integers $a,b_j$ and suppose $\widetilde{D} \neq \Gamma_j$, $\forall 1 \leq j \leq 8$. Then $0 \leq (\widetilde{D} \cdot \Gamma_j)=-2b_j$ so $b_j \leq 0$ for all $j$. Since $(\widetilde{D})^2=-2 \sum_{j=1}^8 b^2_j \geq -2$ by \cite[Prop.\ VIII 3.6]{barth}, there is at most one $b_j$ such that $b_j \neq 0$, and in this case $b_j=-1$. Suppose firstly that all $b_j=0$. Then $\widetilde{D} \sim E$, since $\widetilde{D}$ is integral, and all effective divisors in $|aE|$ are a sum of $a$ divisors in $|E|$, \cite[Prop.\ 2.6(ii)]{donat}. Next suppose $b_j=-1$. Then $\widetilde{D}=aE-\Gamma_j$ and $a \geq 1$ since $a$ is effective. Thus $\widetilde{D}=\widetilde{\Gamma_j}+(a-1)E$ and $(\widetilde{D})^2=-2$, which implies $a=1$ since $\widetilde{D}$ is a smooth and irreducible rational curve (as the unique effective divisor in the linear system $|\widetilde{D}|$ is integral). 

Thus if $i \geq 2$, $D_i$ is either $E$, $\Gamma_j$ or $\widetilde{\Gamma}_j$, for some $j$. Since $\sum_i D_i=L-2E$, we see  $D_1=L-(2+m')E-\sum_{j=1}^8 n'_{1,j} \Gamma_j-\sum_{j=1}^8 n'_{2,j} \widetilde{\Gamma}_j$ for nonnegative integers $m'$ and $n'_{1,j}$, $n'_{2,j}$, $1 \leq j \leq 8$. Since $E=\Gamma_j+\widetilde{\Gamma}_j$, we may rewrite $D_1$ in the form $D_1=L-(2+m)E-\sum_{j=1}^8 (n_{1,j} \Gamma_j+n_{2,j} \widetilde{\Gamma}_j)$, where $m$, $n_{1,j}$, $n_{2,j}$ are nonnegative integers and if $n_{1,j_1} \neq 0$ for some $j_1$ then $n_{2,j_1}=0$, and likewise if $n_{2,j_2} \neq 0$ then $n_{1,j_2} = 0$.
But then one computes
\begin{align*}
(D_1)^2&=(L-2E-(mE+\sum_{j=1}^8 n_{1,j} \Gamma_j+n_{2,j} \widetilde{\Gamma}_j))^2 \\
&=-4-2\sum_{j=1}^8 (n_{1,j}^2+ n_{2,j}^2) -2(L \cdot mE+\sum_{j=1}^8 n_{1,j} \Gamma_j+n_{2,j} \widetilde{\Gamma}_j) \\
& \leq -4
\end{align*}
which is a contradiction (since $(D)^2 \geq -2$ for any integral curve $D$). This proves that $L-2E$ is non-effective.

We have $(L-E)^2=g-3 \geq 8$ for $g \geq 11$. For $g \geq 11$ and $a'>0$ one has $(g-3)a'-6>0$. From the proof of Lemma \ref{little-lem}, this implies there is no effective $F$ with $(F)^2=0$ and $(F \cdot L-E) \leq 3$. 
\end{proof}

For any smooth curve $C$ and $M \in Pic(C)$ with $\deg(M)=d$ and $h^0(M)=r+1$, let $\nu(M):=d-2r$. The \emph{Clifford index} of $C$ is defined by $$\nu(C):=\text{min}\{\nu(M) \; | \; M \in Pic(C) \text{ with $\deg(M) \leq g-1$, $h^0(M) \geq 2$}  \}.$$ Clifford's Theorem states that $\nu(C) \geq 0$ and $\nu(C)=0$ if and only if $C$ is hyperelliptic.

\begin{lem} \label{lem-xyzw}
Let $D \in \Omega_g$ be an effective divisor with $(D)^2 \geq 0$, and assume $L-D$ is effective and $(L-D)^2 >0$. Then $D=cE$ for some integer $c \geq 0$.
\end{lem}
\begin{proof}
Write $D=xL+yE+\sum_{i=1}^8 z_i \Gamma_i$ for integers $x,y,z_i$.  One has $0 \geq (D-L \cdot E)=(x-1)(L \cdot E)$ so that $x \leq 1$. On the other hand $0 \leq (D \cdot E)=x(L \cdot E)$ so that $x \geq 0$. Thus $x=0$ or $x=1$. Suppose firstly that $x=1$. Then $0 < (D-L)^2=-2\sum_{i=1}^8 z_i^2$ which is a contradiction. Hence $x=0$. Then $0 \leq (D)^2=-2(\sum_{i=1}^8 z_i^2)$ and so $z_i=0$ for all $i$, as required.
\end{proof}

\begin{lem} \label{gon-omega}
Let $g \geq 11$ be odd, let $Y_{\Omega_g}$ and $\{L,E, \Gamma_1, \ldots, \Gamma_8 \}$ be as in Lemma \ref{little-lem2} and let $C \in |L|$ be a smooth curve. Then $\nu(C)=\frac{g+1}{2}-2$.
\end{lem}
\begin{proof}
Consider the line bundle $A:=\mathcal{O}_C(E)$. We have $h^0(A)=2$, since $h^1(L-E)=0$ by Kodaira's vanishing theorem. As $\deg(A)=\frac{g+1}{2}$, we see 
$\nu(C) \leq \frac{g+1}{2}-2$. Suppose for a contradiction that $\nu(C) < \frac{g+1}{2}-2$. From \cite[Lem.\ 8.3]{knut}, there is a smooth and irreducible curve $D \seq Y_{\Omega_g}$ with $0 \leq (D)^2 < \nu(C)+2$, $2(D)^2 < (D \cdot L)$ and $\nu(C)= (D \cdot L)-(D)^2-2$ (the sharp inequalities here are due to the fact that $L$ is primitive). We have $$(D-L \cdot D)=(D)^2-(D \cdot L)=-2-\nu(C) \leq -2$$ and $$(D-L)^2=-6-2\nu(C)-(D)^2+2g > -8-3\nu(C)+2g \geq 0,$$ as $\nu(C) \leq \frac{g+1}{2}-2$ and $g \geq 7$. Hence $L-D$ is effective and $D=cE$ for some $c \geq 0$ by Lemma \ref{lem-xyzw}. As $D$ is represented by a smooth curve, we must have $D=E$. But then $$\nu(C)=(E \cdot L)-(E)^2-2=\frac{g+1}{2}-2,$$ giving a contradiction.
\end{proof}

\begin{lem} \label{gon-omega-2}
Assume $g \geq 11$ is odd and let $Y_{\Omega_g}$ be a K3 surface with $Pic(Y_{\Omega_g}) \simeq \Omega_g$ as in Lemma \ref{little-lem2}. Let $M \in \Omega_{g}$ be an effective line bundle on $Y_{\Omega_g}$ satisfying 
\begin{itemize}
\item $0 \leq (M)^2 < \frac{g+1}{2}$
\item $2 (M)^2 < (M \cdot L)$
\item $\frac{g+1}{2}= (M \cdot L)-(M)^2.$
\end{itemize}
Then $M=E$.
\end{lem}
\begin{proof}
The above inequalities give $(M-L \cdot M) < 0$ and $$(M-L)^2=(M)^2+2g-2-2((M)^2+\frac{g+1}{2}) > 0$$ as $(M)^2 < \frac{g+1}{2} \leq g-3$ for $g \geq 7$. Thus $M=cE$ for $c >0$ by Lemma \ref{lem-xyzw}. The equation $\frac{g+1}{2}= (M \cdot L)-(M)^2$ gives $c=1$.
\end{proof}

\begin{lem} \label{unique-special}
Let $g \geq 11$ be odd and let $Y_{\Omega_g}$ and $\{L,E, \Gamma_1, \ldots, \Gamma_8 \}$ be as in Lemma \ref{little-lem2} and let $C \in |L|$ be a smooth curve. Suppose $A \in Pic(C)$ has $h^0(A)=2$ and $deg(A)= \frac{g+1}{2}$. Then $A \simeq E_{|_C}$.
\end{lem}
\begin{proof}
Let  $A \in Pic(C)$ with $h^0(A)=2$ and $deg(A)= \frac{g+1}{2}$. In particular, $\nu(A)=\nu(C)$. Then by \cite[Thm.\ 4.2]{donagi-morrison}, there is some effective divisor $D \seq Y_{\Omega_g} $ such that that $(D \cdot L) \leq g-1$, $L-D$ is effective, $\dim |D| \geq 1$, $D|_C$ achieves the Clifford index on $C$ and such that there exists a reduced divisor $Z_0 \in |A|$ of length $\frac{g+1}{2}$ with $Z_0 \seq D \cap C$. Further from the proof of \cite[Lemma 4.6]{donagi-morrison} and the paragraph proceeding it, we see $\frac{g+1}{2}= (D \cdot L)-(D)^2,$ using that $\nu(D|_C)= \frac{g+1}{2}-2$ by hypothesis. It then follows that all the conditions of Lemma \ref{gon-omega-2} are satisfied, so that $D \in |E|$. As $(E \cdot C)=\frac{g+1}{2}$ and $Z_0$ is reduced of length $\frac{g+1}{2}$,  we have $Z_0 = D \cap C$. Thus $Z_0 \in |E_{|_C}|$ which forces $A \simeq E_{|_C}$.
\end{proof}

The following lemma may be extracted from work of Mukai, cf.\ \cite[\S3]{mukai-duality}, \cite[Lem.\ 2]{mukai-genus11} (although it never appears in this precise form). Despite the simple proof, this lemma is actually rather fundamental, since it gives an example of  a K3 surface $S$ and a divisor $D \seq S$ such that the K3 surface $S$ can be reconstructed merely from the curve $D$ together with a special divisor $A \in Pic(D)$.
\begin{lem}[Mukai] \label{muk-lem}
Let $g \geq 11$ be odd and let $S$ be a K3 surface with $L, M \in Pic(S)$ such that $L^2=2g-2$, $M^2=0$ and $(L \cdot M)=\frac{g+1}{2}$. Let $D \in |L|$ be smooth and set $A:=M_{|_D}$. Further, let $A^{\dagger}=(L-M)_{|_D}$ be the adjoint of $A$ and set $k:=h^0(A^{\dagger})-1$. Assume $L-M$ is very ample and that there is no integral curve $F \seq S$ with $(F)^2=0$ and $(F \cdot L-M)=3$. Assume further that $M$ is represented by an integral curve and that $L-2M$ is not effective. Then $A^{\dagger}$ is very ample, and $S$ is the quadric hull of the embedding $\phi_{A^{\dagger}}: D \hookrightarrow \proj^k$ induced by $A^{\dagger}$. In particular, the linear system $|L|$ contains only finitely many curves $D'$ such that we have an isomorphism $\phi :  D' \simeq D$ with $\phi^*A \simeq M_{|_{D'}}$.
\end{lem}
 \begin{proof}
 As there exists an integral curve in $|M|$, we know $h^1(M)=0$. From the exact sequence
 $$ 0 \to M^* \to L \otimes M^* \to  A^{\dagger} \to 0 $$
 we see that restriction induces an isomorphism $H^0(S,L\otimes M^*) \simeq H^0(D,A^{\dagger})$, which implies that $A^{\dagger}$ is very ample (as $L \otimes M^*$ is) and that $\phi_{A^{\dagger}}: D \hookrightarrow \proj^k$ is the composition of $D \hookrightarrow S$ and $\phi_{L-M}: S \hookrightarrow \proj^k$. Now let $Q \seq \proj^k$ be a quadric containing $\phi_{A^{\dagger}}(D)$. Then we claim that $Q$ contains $S$. Indeed, otherwise there is some effective divisor $T \seq S$ such that $D+T=Q \cap S$ so that $D+T \in |2L-2M|$. But then $T \in |L-2M|$, contradicting that $L-2M$ is not effective by assumption. Thus the quadric hull of $\phi_{A^{\dagger}}(D)$ coincides with the quadric hull of $S \hookrightarrow \proj^k$. But the quadric hull of $S$ is simply $S$ from \cite[Thm.\ 7.2]{donat}. The assumption $g \geq 11$ is needed to ensure $(L-M)^2 \geq 8$ which is required in Saint-Donat's theorem (note that since $L-M$ is very ample by assumption it is not hyperelliptic in the sense of \cite[\S 4.1, Thm.\ 5.2]{donat}). Finally, by the argument above, we have that, for any curve $D' \in |L|$ such that there exists an isomorphism $\phi :  D' \simeq D$ with $\phi^*A \simeq M_{|_{D'}}$, there is an automorphism $f: S \to S$ induced by an automorphism of $\proj^k$ such that $f(D')=D$. Since $S$ is a K3 surface, we deduce that there are only finitely many such $D'$.
 \end{proof}
 
 Putting all the pieces together, we have the following proposition:
\begin{prop} \label{good-dim-prop}
Let $g \geq 11$ be odd. Let $T$ be a smooth and irreducible scheme with base point $0 \in T$. Let $\mathcal{Z} \to T$ be a flat family of K3 surfaces together with an embedding $\phi: \mathcal{Z} \hookrightarrow T \times \proj^g$. Assume $\phi_0: \mathcal{Z}_0 \to \proj^{g}$ is the embedding $ Y_{\Omega_g} \to \proj^g$ induced by $|L|$, where $ Y_{\Omega_g}$ is as in Lemma \ref{little-lem2}. Let $H \seq \proj^g$ be a hyperplane and $C:= H \cap Y_{\Omega_g}$ a smooth curve. Assume that for all $t \in T$, we have a smoothly varying family of hyperplanes $H_t$ with $H_0=H$ and $\mathcal{Z}_t \cap H_t \simeq C$. Then for $t \in T$ general there is an isomorphism $ \psi_t: \mathcal{Z}_t \simeq Y_{\Omega_g}$ and an automorphism $f_t \in Aut(\proj^g)$ such that $f_t \circ \phi_t = \phi_0 \circ \psi_t$.  Furthermore, $f_t$ sends $H_t$ to $H$.
\end{prop}
\begin{proof}
For $t \in T$ general we have a primitive embedding $j: Pic(\mathcal{Z}_t) \hookrightarrow Pic(Y_{\Omega_g})$ with $j(\mathcal{O}_{\mathcal{Z}_t}(1))=L$. By hypothesis, each fibre $\mathcal{Z}_t$ contains a curve isomorphic to $C$ 
which has Clifford index $\frac{g+1}{2}-2$ by Lemma \ref{gon-omega}. By \cite[Lem.\ 8.3]{knut} there is some smooth and irreducible divisor $M' \in Pic(\mathcal{Z}_t)$ satisfying
$0 \leq (M')^2 < \frac{g+1}{2}$, $2 (M')^2 < (M' \cdot \mathcal{O}_{\mathcal{Z}_t}(1))$ and $\frac{g+1}{2}= (M' \cdot \mathcal{O}_{\mathcal{Z}_t}(1))-(M')^2$. As in the proof of Lemma \ref{gon-omega}, these conditions ensure that $\mathcal{O}_{\mathcal{Z}_t}(1)-M'$ is effective. Moreover, $M:=j(M')$ satisfies the conditions of Lemma \ref{gon-omega-2}, so that $M=E$. Thus $(M')^2=0$ and hence $M'$ is a smooth elliptic curve. By lemmas \ref{little-lem} and \ref{little-lem2}, $L-E$ is very ample,  $(L-E)^2 \geq 8$, $L-2E$ is not effective and there exists no effective $F$ with $(F)^2=0$ and $(F \cdot L-E)=3$; thus the same holds for $\mathcal{O}_{\mathcal{Z}_t}(1)-M' \in Pic(\mathcal{Z}_t)$ for $t$ close to $0$. Thus by Lemma \ref{muk-lem} and since $(\mathcal{O}_{\mathcal{Z}_t}(1)-M')_{|_C} \simeq K_C(E^*_{|_C})$  by Lemma \ref{unique-special} we have $ \mathcal{Z}_t \simeq Y_{\Omega_g}$. Furthermore, the embedding  $C \hookrightarrow Z_t$ is identified under this isomorphism with the natural embedding of $C$ into the quadric hull of the embedding $\phi_{K_C(E^*_{|_C})}: C \hookrightarrow \proj^k$. In particular, the embedding is independent of $t$ (up to the action of projective transformations on $\proj^g$). Thus for the general $t$ there is some $f_t \in Aut(\proj^g)=Aut(|L|)$ with $f_t(Z_t)=Y_{\Omega_g}\seq \proj^g$ such that $f_t$ sends $H_t$ to $H$.
\end{proof}
As a direct consequence we have the following corollary.
\begin{cor} \label{bij-cor}
Let $g \geq 11$ be odd and let $Y_{\Omega_g}$ and $\{L,E, \Gamma_1, \ldots, \Gamma_8 \}$ be as in Lemma \ref{little-lem2} and let $C \in |L|$ be a smooth curve. Then the fibre of the morphism $\eta: \mathcal{T}_g^0 \to \mathcal{M}_{g}$ over $[C]$ is zero-dimensional at $[(i:C \hookrightarrow Y_{\Omega_g}, L)]$.
\end{cor}
\begin{remark}
In fact, $\eta : \mathcal{T}_g^0 \to \mathcal{M}_{g}$ is birational onto its image for $g \geq 11$, $g \neq 12$, \cite{ciliberto-lopez-miranda}. Also see \cite[\S 10]{mukai-nonabelian}, \cite{sernesi-brill} for an approach in the odd genus case which more closely resembles the above.
\end{remark}

\section{Generic finiteness of the morphism $\eta$} \label{finny}
In this section we will investigate the generic finiteness of the morphism of stacks
$$\eta \; : \; \mathcal{T}^n_{g,k}  \to \mathcal{M}_{p(g,k)-n}$$ defined by sending $[(f: B \to X, L)]$ to $[B]$. 
Then $\eta$ can be extended to a morphism of stacks
$$\mathcal{W}^{n}_{g,k} \to \overline{\mathcal{M}}_{p(g,k)-n},$$ by sending a pair $[(f: B \to X, L)]$ to $[\hat{B}]$, where $\hat{B}$ denotes the stabilization of the nodal curve $B$; this works in families from the proof of \cite[Prop.\ 2.1]{knudsen}(see also \cite[\S 1.3]{fulpar}). By abuse of notation we will continue to denote this extension by $\eta$.

We start by recalling the basic deformation theory of stable maps. For the construction of $\mathcal{W}^{n}_{g,k}$ as a Deligne--Mumford stack and its elementary deformation theory, we recommend \cite[\S 10]{arakol}. The following is \cite[Prop.\ 4.1]{kemeny-thesis}. Also see \cite[\S 2]{huy-kem} and \cite[Thm.\ 2.4, Rem.\ 3.1]{kool-thomas}:
\begin{prop}
Each component of $\mathcal{W}^{n}_{g,k}$ has dimension at least $p(g,k)-n+19$.
\end{prop}
The criterion below has been used several times in the literature, see \cite{bogomolov}, \cite{huy-kem}.
\begin{prop} \label{ishi-little}
Let $[(f: B \to X, L)] \in \mathcal{W}^{n}_{g,k}$ represent an unramified stable map such that $h^0(N_f) \leq p(B)$, where $p(B)=p(g,k)-n$ denotes the arithmetic genus of $B$. Then for every irreducible component $J \seq \mathcal{W}^{n}_{g,k}$ containing $[(f: B \to X, L)]$ the projection $\pi: J \to \mathcal{B}_g$ is dominant.
\end{prop}
\begin{proof}
 The fibre $\pi^{-1}([(X,L)])$ may be identified with the space of stable maps into the \emph{fixed} surface $X$, and thus each component of $\pi^{-1}([(X,L)])$ containing $[(f: B \to X, L)]$ has dimension at most $h^0(N_f) \leq p(B)$, \cite[\S 3.4.2]{sernesi-def}. Thus $\dim \pi(J)=\dim \mathcal{B}_g$, so $\pi: J \to \mathcal{B}_g$ is dominant (note that this also forces the equality $h^0(N_f) = p(B)$).
\end{proof}

The following is a generalization of \cite[Prop.\ 2.3]{huy-kem}. \footnote{There is a minor mistake in the proof of \cite[Prop.\ 2.3]{huy-kem}; the claimed isomorphism $\Omega_{D_{|_{D_n}}} \simeq \mathcal{O}(-1)$ should be replaced with $det(\Omega_{D_{|_{D_n}}}) \simeq \mathcal{O}(-1)$, as the sheaf $\Omega_{D_{|_{D_n}}}$ is not torsion free.}
\begin{lem} \label{stablem}
Let $f: B \to X$ be an unramified morphism from a connected nodal curve to a K3 surface, and let $N_f$ denote the normal bundle of $f$. Assume that the irreducible components $Z_1, \ldots, Z_s$ of $B$ are smooth. Assume further that we may label the components such that $\bigcup_{i=1}^j Z_i$ is connected for all $j \leq s$. Then $h^0(N_f) \leq p(B)$, where $p(B)$ denotes the arithmetic genus of $B$.
\end{lem}
\begin{proof}
We will prove this by induction on $s$. If $B$ is irreducible, then by assumption $B$ is smooth, so we have a short exact sequence of vector bundles
$$0 \to T_B \to f^* T_X \to N_f \to 0 $$
and taking determinants gives $N_f \simeq \omega_B$. Thus $h^0(N_f)=h^0(\omega_B)=p(B)$. Now let $T:=\overline{B \setminus Z_s}=\bigcup_{i=1}^{s-1} Z_i$; this is connected by assumption. Let $\{p_1, \ldots, p_r\} = Z_s \cap T$. We have a short exact sequence
$$0 \to {N_f}_{|_T}(-p_1 -\ldots -p_r) \to N_f \to {N_f}_{|_{Z_s}} \to 0 .$$ If $A \seq B$ is a connected union of components, and $f_A:=f_{|_A}$, $Y:=A \cap (\overline{B \setminus A})$, then $N_{f_A}(Y)={N_f}_{|_A}$, from \cite[\S 2]{ghs-rat}. Thus
$$ h^0(N_f) \leq h^0(N_{f_T})+h^0(\omega_{Z_s}(p_1+\ldots+p_r)).$$ 
By induction, $h^0(N_{f_T}) \leq p(T)$, and further $h^1(\omega_{Z_s}(p_1+\ldots+p_r))=h^0(\mathcal{O}_{Z_s}(-p_1 \ldots -p_r))=0$, so Riemann--Roch gives $h^0(\omega_{Z_s}(p_1+\ldots+p_r))=p(Z_s)+r-1$. Thus the claim follows from
 $p(B)=p(T)+p(Z_s)+r-1$.
\end{proof}
\begin{remark} \label{slightgen-nodal}
It follows from the proof that the above result may be generalized as follows. Suppose $f: B \to X$ be an unramified morphism from a connected nodal curve to a K3 surface, and $B= \bigcup_{i=1}^s Z_i$ where $Z_1$ is connected, but not necessarily irreducible or smooth, and with $Z_2, \ldots, Z_s$ smooth (and with $s>1$). Assume $\bigcup_{i=1}^j Z_i$ is connected for all $j \leq s$, and $h^0(N_{f_1}) \leq p(Z_1)$, where $f_1:=f_{|_{Z_1}}$. Then $h^0(N_f) \leq p(B)$.
\end{remark}

We now need the following result on simultaneous normalization of families of singular curves, which in this generality is usually attributed to Raynaud, generalising the results of Teissier, \cite{teissier}. For modern treatments and further generalisations, see \cite{chiang-lipman},  \cite[Thm.\ 12]{kollar-simultaneous}.
\begin{prop} [Teissier, Raynaud] \label{sim-res}
Let $B$ be a normal, integral scheme of finite-type over $\C$, and let 
$$ f: \mathcal{C}_1 \to B$$ be a projective, flat family of relative dimension one with reduced fibres. Then there exists a simultaneous resolution of $f$, i.e.\ a flat family $ f': \mathcal{C}_2 \to B$ of relative dimension one with normal fibres, together with a finite map $h: \mathcal{C}_2 \to \mathcal{C}_1$ such that $f \circ h=f'$ and the restriction morphism $$h_b:  \mathcal{C}_{2,b} \to \mathcal{C}_{1,b} $$ is the normalization map for each fibre over $b \in B$.
\end{prop}
\begin{lem} \label{defo-nodal-lemma}
Let $f: B \to X$ be an unramified morphism from an integral, nodal curve $B$ to a K3 surface, with $[(f: B \to X,L)] \in \mathcal{W}^{n}_{g,k}$ and assume that $f$ is birational onto its image. Then $[(f: B \to X,L)]$ lies in the closure of $\mathcal{T}^n_{g,k}$.
\begin{proof}
Suppose there was a component of $J$ of $\mathcal{W}^{n}_{g,k}$ containing $[(f: B \to X,L)]$ such that if  $[(f': B' \to X',L')]$ is general, then $B'$ is nodal with at least $m >0$ nodes. Replace $J$ with the dense open subset parametrizing unramified maps, which are birational onto the image and such that the base $B'$ is integral with exactly $m$ nodes. Composing $f'$ with the normalization $\widetilde{B} \to B'$ gives an unramified stable map $h: \widetilde{B} \to X'$; thus we have a map $G: J(\C) \to \mathcal{T}^{n+m}_{g,k}(\C)$ between the \emph{sets} of closed points of the respective stacks. The fact that $f$ is birational onto its image implies that for a general $y \in Im(G)$, $G^{-1}(y)$ is a finite set. Indeed, if $y$ corresponds to the stable map $h: \widetilde{B} \to X'$, then any element of $G^{-1}(y)$ corresponds to an unramified stable map $h': D \to X'$, birational onto its image, where $D$ is integral and has exactly $m$ nodes and the normalization of $D$ is $ \widetilde{B}$. Further, 
if $\mu: \widetilde{B} \to D$ is the normalization morphism, then $h' \circ \mu=h$, by definition of $G$. As $h'$ is birational onto its image, the set 
$$ Z := \{ z \in \widetilde{B} \; | \; \text{There exists $y \neq z \in \widetilde{B}$ with $h(z)=h(y)$} \}$$
is finite. Since $h'$ is obtained from $h$ by glueing $m$ pairs of points in $Z$, there are only finitely many possibilities for $h'$.

We claim that, at least after a finite base change, $G$ is locally induced by a morphism of stacks. After replacing $J$ with an etale cover, we may assume it comes with a universal family. Let $J' \to J$ be the normalization of $J$. Pulling back the universal family on $J$ gives a family of stable maps over $J'$. In particular, we have a flat family $\mathcal{B} \to J'$ of nodal curves with exactly $m$ nodes specializing to $B'$.  Applying Proposition \ref{sim-res}, we may simultaneously resolve the $m$ nodes of the fibres of $\mathcal{B}$ to produce a family $\widetilde{\mathcal{B}} \to J'$ of \emph{smooth} curves, together with a morphism $\widetilde{\mathcal{B}} \to \mathcal{B}$ restricting to the normalization over each point in $J'$. By composing with the universal family of stable maps $\mathcal{B} \to \mathcal{X}$, where $\mathcal{X}$ is a family of K3 surfaces, we produce a family of stable maps $\widetilde{\mathcal{B}} \to \mathcal{X}$ over $J'$. By the universal property of $\mathcal{T}^{n+m}_{g,k}$, this produces a morphism of Deligne--Mumford stacks $J' \to \mathcal{T}^{n+m}_{g,k}$ which coincides with the composition $G'$ of $G$ with $J' \to J$ on the level of closed points. As this morphism is generically finite onto its image, the dimension of $\mathcal{T}^{n+m}_{g,k}$ is at least $\dim J'=\dim J \geq p(B)+19$. But $\mathcal{T}^{n+m}_{g,k}$ is smooth of dimension $p(B)-m+19$, so this is a contradiction.
\end{proof}
\end{lem}

\begin{prop} \label{prim-cor}
Let $[(f: B \to X, L)] \in \mathcal{W}^{n}_{g,k}$ represent an unramified stable map such that  $h^0(N_f) \leq p(B)$, where $p(B)=p(g,k)-n$. Assume furthermore that there is no decomposition $$B= \bigcup_{i =1}^t B_i$$ for $t >1$ with each $B_i$ a connected union of irreducible components of $B$ such that $B_i$ and $B_j$ meet transversally (if at all) for all $i \neq j$ and such that for all $1 \leq i \leq t$, $f_*(B_i) \in |m_iL|$ for a positive integer $m_i >0$ (this is automatic if $k=1$). Lastly, assume that there is some component $B_j$ such that $f_{|_{B_j}}$ is birational onto its image, and if $B_i \neq B_j$ is any component, $f(B_i)$ and $f(B_j)$ intersect properly. Then $[(f: B \to X, L)]$ lies in the closure of $\mathcal{T}^{n}_{g,k}$.
\end{prop}
\begin{proof}
 We need to show that we may deform $[(f: B \to X, L)]$ to an unramified stable map $[(f': B' \to X', L')]$ with $B'$ irreducible and smooth. We will firstly show that $[(f: B \to X, L)]$ deforms to a stable map with irreducible base. 
 
 From Proposition \ref{ishi-little}, $\pi: \mathcal{W}^{n}_{g,k} \to \mathcal{B}_g$ is dominant near $[(f: B \to X, L)]$. Thus we may deform $[(f: B \to X, L)]$ to an unramified stable map $[(f': B' \to X', L')]$, where $Pic(X') \simeq \mathbb{Z}L'$. For any irreducible component $Z \seq B'$, $f'_*(Z) \in |a_iL'|$ for some integer $a_i >0$. Given any one-parameter family $\gamma(t)=[(f_t: B_t \to X_t, L_t)]$ of stable maps with $\gamma(0)=[(f: B \to X, L)]$, then, by the Zariski connectedness theorem, after performing a finite base change about $0$ if necessary the irreducible components of $B_t$ for generic $t$ deform to a connected union of irreducible components of $B$ as $t \to 0$. Thus the condition on $B$ ensures that $[(f: B \to X, L)] $ deforms to an unramified stable map of the form $[(f': B' \to X', L')]$ with $B'$ \emph{integral} and nodal. 
 
 We will next show that $[(f: B \to X, L)] $ deforms to an unramified stable map of the form $[(f': B' \to X', L')]$ with $B'$ integral and such that $f'$ is birational onto its image. By Lemma \ref{defo-nodal-lemma}, this will complete the proof. Let $S$ be a smooth, irreducible, one-dimensional scheme with base point $0$, and suppose we have a diagram
 \[
\xymatrix{
\mathcal{B} \ar[r]^{\tilde{g}}  \ar[rd]_{\pi_1} & \mathcal{X} \ar[d]^{\pi_2} \\
&S
}
\] 
with $\tilde{g}$ proper, $\pi_1$, $\pi_2$ flat and with $\tilde{g}_s: \mathcal{B}_s \to \mathcal{X}_s$ an unramified stable map to a K3 surface for all $s$, with $\tilde{g}_0=f$ and such that $\mathcal{B}_s$ is integral for $s \neq 0$. There is an an $S$-flat line bundle $\mathcal{L}$ on $\mathcal{X}$, with $\mathcal{L}_0=kL$ and that the cycle $\tilde{g}_{*}(\mathcal{B}) \sim \mathcal{L}$ is a relatively effective (Cartier) divisor. So the cycle $\tilde{g}_{*}(\mathcal{B})$ may be considered as an $S$-relatively effective divisor $\bar{\mathcal{B}} \seq \mathcal{X}$. By the assumptions on $f=\tilde{g}_0$, the irreducible surface $\bar{\mathcal{B}}$ is reduced on a dense open subset meeting $f(D_j)$, which forces $\tilde{g}_s$ to be birational for $s$ near $0$ (as if $\deg(\tilde{g}_s)=d$, $\bar{\mathcal{B}}_s=d \tilde{g}_s(\mathcal{B}_s)$).

\end{proof}

We next aim to reduce the study of generic finiteness of $\eta$ to that of $\eta$ for $m >>n$. 
\begin{lem} \label{red-largen}
Let $n \leq m$ with $p(g,k)-m>0$. Assume that there is a component $I_m \seq \mathcal{T}^m_{g,k}$ such that 
$${\eta}_{|_{I_m}}: I_m \to \mathcal{M}_{p(g,k)-m}$$ is generically finite and assume that for the general $[f: B \to X] \in I_m$, $B$ is non-trigonal. Then there exists a component $I_n \seq \mathcal{T}^n_{g,k}$
such that  $${\eta}_{|_{I_n}}: I_n \to \mathcal{M}_{p(g,k)-n}$$ is generically finite. For the general $[f': C \to Y] \in I_n$, $C$ is non-trigonal.
\end{lem}
\begin{proof}
Let $[(f: B \to X,L)] \in I_m$ be a general point. By \cite[Thm.\ B]{ded-sern}, $f(B)$ is a nodal curve (with precisely $m$ nodes). Thus there is an integral, nodal curve $B'$ with $m-n$ nodes with normalization $\mu: B \to B'$ such that $f$ factors through a morphism $\tilde{g}: B' \to X$. The fibre of $\eta$ over the stable curve $[B'] \in \overline {\mathcal{M}}_{p(g,k)-n}$ is zero-dimensional near $[(\tilde{g},L)]$ as otherwise we could compose with $\mu$ to produce a one dimensional family near $[(f: B \to X,L)]$ in the fibre of $\eta$ over $[B]$. Since $[(\tilde{g},L)]$ lies in the closure of $\mathcal{T}^n_{g,k}$ by Lemma \ref{defo-nodal-lemma}, we see that there exists a component $I_n \seq \mathcal{T}^n_{g,k}$
such that  $${\eta}_{|_{I_n}}: I_n \to \mathcal{M}_{p(g,k)-n}$$ is generically finite. As the normalization of $B'$ is non-trigonal by assumption, $B'$ must lie outside the image of 
$$\overline{\mathcal{H}}_{3, p(g,k)-n} \to \overline{\mathcal{M}}_{p(g,k)-n}, $$
in the notation of \cite[Thm.\ 3.150]{harris-morrison}. Thus, for the general $[f': C \to Y] \in I_n$, $C$ is non-trigonal.
\end{proof}

The following proposition gives a criterion for generic finiteness of the morphism
$$\eta: \mathcal{T}^n_{g,k} \to \mathcal{M}_{p(g,k)-n}$$ on one component $I$. The idea is to assume we have an unramified map $f_0: C_0 \to X$ representing a point in $\mathcal{W}^{n'}_{g'}$ such that finiteness of the moduli map holds near the point representing $f_0$. If we then build a new morphism $f: C_0 \cup \proj^1 \to X$ by finding a rational curve $f(\proj^1)$ in $X$, and if we further assume $C_0 \cup \proj^1$ is a stable curve (i.e.\ $\proj^1$ intersects $C_0$ in at least three points), then by rigidity of rational curves in $X$, one sees easily that finiteness of $\eta$ holds near the point representing $f$ in $\mathcal{W}^n_{g,k}$, where $n,k$ are such that $f$ represents a point in $\mathcal{W}^n_{g,k}$.
\begin{prop} \label{finiteness-criterion}
Assume there exists a polarized K3 surface $(X,L)$ and an unramified stable map $f: B \to X$ with
$f_*(B) \in |kL|$. Assume:
\begin{enumerate}
\item $[(f: B \to X, L)]$ lies in the closure of $\mathcal{T}^{n}_{g,k}$.
\item There exists an integral, nodal component $C \seq B$ of arithmetic genus $p' \geq 2$ such that $f_{|_C}$ is an unramified morphism $j: C \to X$, birational onto its image. Let $k'$ be an integer such that there is a big and nef line bundle $L'$ on $X$ with $j_*(C) \in |k'L'|$, and let $g'=\frac{1}{2} (L')^2+1$.
\item The fibre of the morphism $\eta: \mathcal{W}^{n'}_{g',k'} \to \overline{\mathcal{M}}_{p'}$ over $[C]$ is zero-dimensional near $[(j: C \to X, L')]$, where $n'=p(j(C))-p'$.
\item If $\tilde{C}$ is the normalization of $C$, then $\tilde{C}$ is non-trigonal.
\item If $D \seq B$ is a component, $D \neq C$, then $D$ has geometric genus zero.
\item The stabilization morphism $B \to \hat{B}$ is an isomorphism in an open subset $U \seq B$ such that $C \seq U$.
\end{enumerate} 
Then there exists a component $I \seq \mathcal{T}^n_{g,k}$ such that ${\eta}_{|_I}$ is generically finite and $[(f: B \to X, L)]$ lies in the closure of $I \seq \mathcal{W}^{n}_{g,k}$. For the general $[f': B'\to X'] \in I$, $B'$ is non-trigonal.
\end{prop}
\begin{proof}
We have a morphism $\eta: \mathcal{W}^{n}_{g,k} \to \overline{\mathcal{M}}_{p(g,k)-n}$. By assumption $1$, it suffices to show $\eta^{-1}([\hat{B}])$ is zero-dimensional near $[(f: B \to X, L)]$. Let $S$ be a smooth, irreducible, one-dimensional scheme with base point $0$, and suppose we have a commutative diagram
 \[
\xymatrix{
\mathcal{B} \ar[r]^{\tilde{g}}  \ar[rd]_{\pi_1} & \mathcal{X} \ar[d]^{\pi_2} \\
&S
}
\] 
with $\tilde{g}$ proper, $\pi_1$, $\pi_2$ flat and with $\tilde{g}_s: \mathcal{B}_s \to \mathcal{X}_s$ an unramified stable map to a K3 surface for all $s$, with $\tilde{g}_0=f$. Further assume $\hat{\mathcal{B}_s} \simeq \hat{B}$. For any $s \in S$, we have open subsets $U_s \seq \mathcal{B}_s$, $V_s \seq \hat{B}$ with the stabilization map inducing isomorphisms $U_s \simeq V_s$ and such that $\hat{B} \setminus V_s$ has zero-dimensional support. By assumption $6$, $C \seq V_0$, and thus for $s$ near $0$, $C \seq V_s \simeq U_s$. Thus, after performing a finite base change about $0 \in S$, there exists an irreducible component $\mathcal{C} \seq \mathcal{B}$, such that $\mathcal{C}_s \simeq C$, for all $s$ near $0$. We have a commutative diagram
 \[
\xymatrix{
\mathcal{C} \ar[r]^{h}  \ar[rd]_{{\pi_1}_{|_{\mathcal{C}}}} & \mathcal{X} \ar[d]^{{\pi_2}} \\
&S
}
\]
where ${\pi_1}_{|_{\mathcal{C}}}$ is flat and $h:=\tilde{g}_{|_{\mathcal{C}}}$. Since $h_0=j$, assumption $3$ gives $\mathcal{X}_s \simeq X$ and $h_s=j$ for all $s$. In particular, $\tilde{g}_s$ is a one-dimensional family of unramified morphisms into a \emph{fixed} K3 surface. From assumption $5$, the fact that rational curves on a K3 surface are rigid and since $\tilde{g}_s$ is unramified, we conclude that $\mathcal{B}_s \simeq B$ and $\tilde{g}_s: B \to X$ is independent of $s$. Thus  $\eta^{-1}(\hat{B})$ is zero-dimensional near $[(f: B \to X, L)]$. Hence there exists a component $I \seq \mathcal{T}^n_{g,k}$ such that ${\eta}_{|_I}$ is generically finite and $[(f: B \to X, L)]$ lies in the closure of $I \seq \mathcal{W}^{n}_{g,k}$. Since $\tilde{C}$ is non-trigonal, the stabilization $\hat{B}$ must lie outside the image of 
$$\overline{\mathcal{H}}_{3, p(g,k)-n} \to \overline{\mathcal{M}}_{p(g,k)-n}.$$ Thus, for the general $[f': B'\to X'] \in I$, $B'$ is non-trigonal.
\end{proof}

We will apply the above criterion to prove generic finiteness of $\eta$ on one component, for various bounds on $p(g,k)-n$. We first consider the case $k=1$. To begin, we will need an easy lemma. Let $p >h \geq 8$ be integers, and let $l,m$ be nonnegative integers with
$$p-h= \left \lfloor \frac{h+1}{2} \right \rfloor l+m $$
and $0 \leq m < \left \lfloor \frac{h+1}{2} \right \rfloor $. Define:
\begin{align*}
s_1:= \begin{cases}
   p-h-1, & \text{if $m=0$ or $m=\left \lfloor \frac{h+1}{2} \right \rfloor -1$}\\
    p-h+1, & \text{otherwise}.
  \end{cases}
\end{align*} 
Let $P_{p,h}$ be the rank three lattice generated by elements $\{M, R_1, R_2 \}$ and with intersection form given with respect to this ordered basis by:
\[ \left( \begin{array}{ccc}
2h-2 & s_1 & 3 \\
s_1 & -2 & 0 \\
3 & 0 & -2 \end{array} \right).\]
The lattice $P_{p,h}$ is even of signature $(1,2)$.
\begin{lem} \label{onenodal-lem-a}
Let $p>h \geq 8$. There exists a K3 surface $S_{p,h}$ with $Pic(S_{p,h}) \simeq P_{p,h}$ as above such that the classes $M, R_1, R_2 $ are each represented by integral curves and with $M$ very ample. If $h $ is odd and at least $11$, then for $D \in |M|$ general the fibre of $\eta: \mathcal{T}_h^0 \to \mathcal{M}_{h}$ is zero-dimensional at $[(i:D \hookrightarrow S_{p,h}, M)]$ and $D$ is non-trigonal. Furthermore, any divisor of the form $-xM+yR_1+zR_2 $ for integers $x,y,z$ with $x>0$ is not effective.
\end{lem}
\begin{proof}
Consider the K3 surface $Y_{\Omega_h}$ from Lemma \ref{lem-aaa}. We choose $d_1:=(L \cdot \Gamma_1)$ to be $m+1$ if $0 < m <  \left \lfloor \frac{h+1}{2} \right \rfloor  -1$ and $d_1=\left \lfloor \frac{h+1}{2} \right \rfloor  -1$ if $m=0$ and $d_1=\left \lfloor \frac{h+1}{2} \right \rfloor  -2$ if $m=\left \lfloor \frac{h+1}{2} \right \rfloor -1$. Further set $d_2=3$ and let all other $d_i$ be arbitrary integers in the range $1 \leq d_i < \left \lfloor \frac{h+1}{2} \right \rfloor$. We define a primitive embedding
$j: P_{p,h} \hookrightarrow \Omega_h$ as follows. If $m=0$ (so that $l \geq 1$ as $p>h$), we define the embedding via
$M \mapsto  L$, $R_1 \mapsto (l-1)E+\Gamma_1$, $R_2 \mapsto \Gamma_2$. If $m \neq 0$ we define the embedding via
$M \mapsto  L$, $R_1 \mapsto lE+\Gamma_1$, $R_2 \mapsto \Gamma_2$.
Let $M_{P_{p,h}}$ be the moduli space of ample, $P_{p,h}$-polarized K3 surfaces, \cite[Def.\ p.1602]{dolgachev}. The moduli space $M_{P_{p,h}}$ has dimension $17=19-2$, \cite[Prop.\ 2.1]{dolgachev}.  Let $M_1$ be a component containing the ample, $P_{p,h}$-polarized K3 surface $[Y_{\Omega_h}]$. Then the general point of $M_1$ represents a projective K3 surface $S_{p,h}$ with $Pic(S_{p,h}) \simeq P_{p,h}$, \cite[Cor.\ 1.9, Cor.\ 2.9]{morrison-large}. Further,  $M, R_2 $ are each represented by integral curves and $M$ is very ample by Lemmas \ref{lem-aaa} and \ref{little-lem}. If $h $ is odd and at least $11$, the statement about the fibres of $\eta$ follows from Corollary \ref{bij-cor} by semicontinuity, and the non-trigonality follows from Lemma \ref{gon-omega}. 

We now claim that any divisor of the form $-xM+yR_1+zR_2$ for integers $x,y,z$ with $x>0$ is not effective. By degenerating $S_{p,h}$ to $Y_{\Omega_h}$ as above, it suffices to show that $-xL +y j(R_1)+z\Gamma_2 \in Pic(Y_{\Omega_h})$ is not effective. But this is clear, since the rank ten lattice $Pic(Y_{\Omega_h})$ contains the class of a smooth, integral, elliptic curve $E$ with $( E \cdot -xL +y j(R_1)+z\Gamma_2)=-x (E \cdot L) <0$.

The $-2$ class $R_1$ is effective since $(R_1)^2=-2$, $(R_1 \cdot M) >0$. It remains to show that $R_1$ is integral. Let $D$ be any integral component of $R_1$ with $(D)^2=-2$, $(D \cdot R_1)<0$. Writing $D= xM+yR_1+zR_2$ we see that $x=0$ (as $R_1-D$ is effective). Thus $(D \cdot R_1) <0$ implies $y>0$ and $-2=(D)^2=-2(y^2+z^2)$ forces $z=0$; thus $D=R_1$ is integral.
\end{proof}
We can now prove Theorem \ref{finiteness}:
\begin{proof} [Proof of Theorem \ref{finiteness}]
For $g \geq 11$ and $n \geq 0$, we let $0 \leq r(g) \leq 5$ be the unique integer such that 
$g-11 = \left \lfloor \frac{g-11}{6} \right \rfloor 6 +r(g)$.
Set $l_g:=12$ if $r(g)=0$, $l_g:=13$ if $1 \leq r(g) <5$ and $l_g:=15$ if $r(g)=5$.
We want to show that there is a component $I \seq \mathcal{T}^n_{g}$ such that $${\eta}_{|_I}: I \to \mathcal{M}_{g-n}$$ is generically finite for
$g-n \geq l_g$. Furthermore, we will prove that for the general $[\tilde{f}: C' \to X'] \in I$, $C'$ is non-trigonal.

From Lemma \ref{red-largen} it suffices to prove the result for the maximal value of $n$. Assume $g-n = l_g$ if $r(g) \neq 0$ and $g-n = 15$ if $r(g)=0$. 
Set $p=g$, $h=11$ and consider the lattice $P_{g,11}$ and K3 surface $S_{g,11}$ from Lemma \ref{onenodal-lem-a}. Set $\epsilon=1$ if $r(g)=0$ or $r(g)=5$ and $\epsilon=0$ otherwise. Then $(M+R_1+\epsilon R_2)^2=2g-2$ and $M+R_1+\epsilon R_2$ is ample. Let $D \in |M|$ be general and consider the curve $D_1=D \cup R_1 \cup \epsilon R_2$ where all intersections are transversal. Choose any subset of $s_1-3$ distinct points of $D \cap R_1$ and let $f: B \to D_1$ be the partial normalization at the chosen points. Then $B$ has arithmetic genus $l_g$ if $r(g) \neq 0$ and genus $15$ if $r(g)=0$. Further, $f$ satisfies the conditions of Proposition \ref{finiteness-criterion}. Then there is a component $I \seq \mathcal{T}^n_{g}$ such that $${\eta}_{|_I}: I \to \mathcal{M}_{g-n}$$ is generically finite for
$g-n \geq l_g$ if $r(g) \neq 0$ and for $g-n \geq 15$ if $r(g)=0$. For the general $[f: C \to X] \in I$, $C$ is non-trigonal.

We now wish to improve the bound in the case $r(g)=0$.
Recall from Section \ref{mukai} the lattice $\Omega_{11}$ with ordered basis $\{L, E, \Gamma_1, \ldots, \Gamma_8 \}$. Thus the general $C \in |L|$ is a smooth, genus $11$ curve and $(L \cdot E)=6$. Note that $\widetilde{\Gamma}_i \sim E-\Gamma_i$ is a class satisfying $(\widetilde{\Gamma}_i)^2=-2$, $(\widetilde{\Gamma}_i \cdot E)=0$, $(\widetilde{\Gamma}_i \cdot \Gamma_i)=2$. From Lemma \ref{lem-aaa}, $\widetilde{\Gamma}_i$ is represented by an integral class. Further $\widetilde{\Gamma}_i+\Gamma_i$ is an $I_2$ singular fibre of $|E|$ for any $i \geq 3$. We will denote by $x_i$ and $y_i$ the two nodes of $\widetilde{\Gamma}_i+\Gamma_i$.

Set $m:=\left \lfloor \frac{g-11}{6} \right \rfloor$ and assume $r(g)=0$. Consider the primitive, ample line bundle $H:=L+\left \lfloor \frac{g-11}{6} \right \rfloor E$, which satisfies $(H)^2=2g-2$. Let $C \in |L|$ be a general smooth curve which meets $\Gamma_1$ transversally. Let $B$ be the abstract curve which is the union of $C$ with $2m$ copies of $\proj^1$ as in the diagram below.
\\
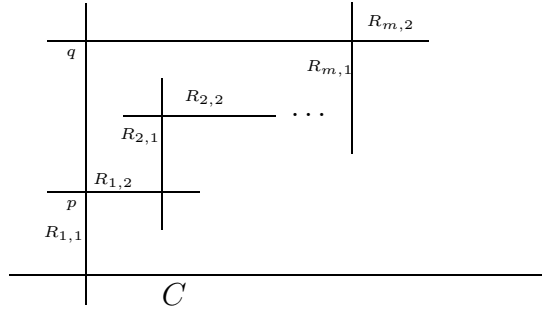
$$
\setlength{\unitlength}{1cm}
\begin{picture}(8,5)
\put(1,0){\line(0,1){4}}
\put(0.5,1.5){\line(1,0){2}}
\put(2,1){\line(0,1){2}}
\put(1.5,2.5){\line(1,0){2}}
\put(3.7,2.5){$\ldots$}
\put(4.5,2){\line(0,1){2}}
\put(0.5,3.5){\line(1,0){5}}

\put(0,0.4){\line(1,0){7}}

\put(2,0){\mbox{$C$}}
\put(0.45,0.9){\tiny\mbox{$R_{1,1}$}}
\put(0.75,1.3){\tiny\mbox{$p$}}
\put(0.75,3.3){\tiny\mbox{$q$}}
\put(1.1,1.6){\tiny\mbox{$R_{1,2}$}}
\put(1.45,2.2){\tiny\mbox{$R_{2,1}$}}
\put(2.3,2.7){\tiny\mbox{$R_{2,2}$}}
\put(3.9,3.1){\tiny\mbox{$R_{m,1}$}}
\put(4.7,3.7){\tiny\mbox{$R_{m,2}$}}
\end{picture}
$$
\\
Thus $B=C \cup R_{1,1} \cup R_{1,2} \ldots \cup R_{m,1} \cup R_{m,2}$, where each $R_{i,j}$ is smooth and rational, all intersections are transversal and described as follows for $m \geq 2$: $R_{i,j} \cap C= \emptyset$ unless $(i,j)=(1,1)$, $R_{1,1}$ intersects $C$ in one point, and $R_{i,j}$ intersects $R_{k,l}$ in at most one point, with intersections occuring if and only if, after swapping $R_{i,j}$ and $R_{k,l}$ if necessary, we have ($i=k$ and $l=j+1$), ($k=i+1$ and $l \neq j$) or ($(i,j)=(1,1)$ and $(j,k)=(m,2)$). The arithmetic genus of $B$ is $12$. For $m=1$, $B=C \cup R_{1,1} \cup R_{1,2}$ where $C$ intersects $R_{1,1}$ transverally in one point and $R_{1,1} \cap C = \emptyset$ and $R_{1,1}$ intsersects $R_{1,2}$ transversally in two points.

There is then a unique unramified morphism $f: B \to Y_{\Omega_{11}}$ such that $f_{|_C}$ is a closed immersion, $f_{|_{R_{i,1}}}$ is a closed immersion with image $\Gamma_1$ and $f_{|_{R_{i,2}}}$ is a closed immersion with image $\widetilde{\Gamma_1}$, for $1 \leq i \leq m$ and where $(f(p),f(q))=(x_1,y_1)$. Thus $f_*{B} \in |H|$. Note that the stabilization $\hat{B}$ has two components, namely it is the union of $C$ with a rational curve with one node. Thus the claim holds in the $r(g)=0$ case by Proposition \ref{finiteness-criterion}, together with Lemma \ref{stablem} and Corollary \ref{bij-cor}. 
\end{proof}
We now turn to the nonprimitive case.
\begin{lem} \label{easynontriglemma}
Let $\Delta$ be an smooth, one-dimensional algebraic variety over $\C$, with $0 \in \Delta$ a closed point. Let $\mathcal{C} \to \Delta$ be a flat family of nodal curves,
with general fibre integral and such that $$ \mathcal{C}_0 = B \cup \Gamma_1 \cup \ldots \cup \Gamma_k,$$
with $B$ smooth and non-trigonal, and $\Gamma_i$ smooth curves for $1 \ldots i \leq k$, with all intersections transversal. Then if $\tilde{\mathcal{C}}_t$ is the normalization of $\mathcal{C}_t$, $t \in \Delta$ general, then $\tilde{\mathcal{C}}_t$ is non-trigonal.
\end{lem}
\begin{proof}
We go through the first steps of the usual stable reduction procedure, \cite[\S X.4]{arbarello-II}. Let $\mu : \tilde{\mathcal{C}} \to \mathcal{C}$ be the normalization of the integral surface $\mathcal{C}$. Then $\mu$ is finite and birational, and further is an isomorphism outside the preimage of singular points in the fibres of $\mathcal{C}$. Further, $\tilde{\mathcal{C}} \to \Delta$ is a flat-family of curves by \cite[Prop.\ 4.3.9]{liu}. As $\tilde{\mathcal{C}}$ has isolated singularities, $\tilde{\mathcal{C}}_t$ must be smooth for $t \in \Delta$ general, and must be the normalization of $\mathcal{C}_t$, as $\mu$ is finite and birational. Since $\mu_0: \tilde{\mathcal{C}}_0 \to \mathcal{C}_0$ is finite, and an isomorphism outside singular points of $\mathcal{C}_0$, the components of $\tilde{\mathcal{C}}_0$ are isomorphic to $B$, $\Gamma_i$. Further, $\mu_0$ factors through the total normalization of $\mathcal{C}_0$ (which equals the disjoint union of the components). This forces $\tilde{\mathcal{C}}_0$ to itself be nodal. Further $\tilde{\mathcal{C}}_0$ is connected, by taking a desingularization of the surface $\tilde{\mathcal{C}}$ and applying \cite[Thm.\ 8.3.16]{liu} followed by \cite[Cor.\ 8.3.6]{liu}. Since $B$ is smooth and non-trigonal, the stabilization of $\tilde{\mathcal{C}}_0$ lies outside the image of $\overline{\mathcal{H}}_{3,p}$, where $p$ is the arithmetic genus of $\tilde{\mathcal{C}}_0$. It follows that $\tilde{\mathcal{C}}_t$ is non-trigonal.
\end{proof}

\begin{lem} \label{hjz}
Let $Z_a$ be a general K3 surface with Picard lattice $\Lambda_a$ generated by elements $D,F,\Gamma$ giving the intersection matrix
\[ \left( \begin{array}{ccc}
2a-2 & 6 & 1 \\
6 & 0 & 0 \\
1 & 0 & -2 \end{array} \right)\]
Assume that $14 \leq a \leq 19$.
Then we may pick the basis such that $D$, $F$, $\Gamma$ are all effective and represented by integral, smooth curves with $D$ ample. Further, there is an unramified stable map $f_a: B_a \to Z_a$, birational onto its image, with $B_a$ an integral, nodal curve of arithmetic genus $13$ for $14 \leq a \leq 15$ and $15$ for $16 \leq a \leq 19$, such that $f_{a*}(B_a) \in |D|$ and $f_a$ satisfies the conditions of Proposition \ref{finiteness-criterion}. Further, there is an integral, rational nodal curve $F_0 \in |F|$ which meets $f_a(B_a)$ transversally, and $\Gamma$ meets $f_a(B_a)$ transversally in one point.
\end{lem}
\begin{proof}
 Let $M_{\Lambda_a}$ be the moduli space of pseudo-ample, $\Lambda_a$-polarized K3 surfaces, \cite{dolgachev}. This has at most two components, which locally on the period domain are interchanged via complex conjugation. Consider the lattice $\Omega_{11}$ with ordered basis $\{L, E, \Gamma_1, \ldots, \Gamma_8 \}$ and set $d_8=1$. For $d_1, d_2 \geq 3$, let $H$ be the primitive, ample line bundle 
 $H=L+\Gamma_1 + \epsilon \Gamma_2$, where $\epsilon=0$ for $14 \leq a \leq 15$ and $\epsilon=1$ for $16 \leq a \leq 19$. Choose $3\leq d_1, d_2 \leq 5$ such that $H^2=2a-2$; it is easily checked that all six possibilities can be achieved. There is a primitive lattice embedding $$\Lambda_a \hookrightarrow \Omega_{11}$$ defined by $D \mapsto H$, $\Gamma \mapsto \Gamma_8$, $F \mapsto E$. 
 
 Let $Y_{\Omega_{11}}$ be any K3 surface with $Pic(Y_{\Omega_{11}}) \simeq \Omega_{11}$, and choose the basis $\{L, E, \Gamma_1, \ldots, \Gamma_8 \}$ as in Lemma \ref{lem-aaa}. Consider the curve $C \cup \Gamma_1 \cup \epsilon \Gamma_2 \in |H|$, where $C \in |L|$ is a general smooth curve. By partially normalizing at all nodes other than three on $C \cup \Gamma_1$ and three on $C \cup \Gamma_2$ (for $\epsilon \neq 0$), we construct an unramified stable map $\bar{f}_a: \bar{B}_a \to Y_{\Omega_{11}}$ with $\bar{f}_{a*} (\bar{B}_a) \sim H$ which is birational onto its image and satisfying the conditions of Proposition \ref{finiteness-criterion}. Note that $\bar{B}_a$ has arithmetic genus $13$ for $14 \leq a \leq 15$ and $15$ for $16 \leq a \leq 19$. After deforming $\bar{f_a}$ to an unramified stable map $f_a: B_a \to Z_a$, we find $B_a$ must become integral, since it is easily checked that $Z_a$ contains no smooth rational curves $R$ with $(R \cdot F)=(R \cdot \Gamma)=0$. Further, the normalization $\tilde{B}_a$ is non-trigonal by Lemma \ref{easynontriglemma}.
 
 Thus the claim on $f_a$ follows from the proof of Proposition \ref{finiteness-criterion}. 
Note that the $I_2$ fibre $\Gamma_7 + \widetilde{\Gamma}_7$ must deform to an integral, nodal, rational curve on $Z_a$, since $Z_a$ contains no smooth rational curves which avoid $F$ and $\Gamma$. If $\Omega_{11} \hookrightarrow Pic(Y_{\Omega_{11}})$ is the embedding as in Lemma \ref{lem-aaa}, and if $Y^c_{\Omega_{11}}$ is the conjugate K3 surface, then we obviously have an embedding $\Omega_{11} \hookrightarrow Pic(Y^c_{\Omega_{11}})$ satisfying the conclusions of Lemma \ref{lem-aaa}. Thus the claim holds for the general pseudo-ample, $\Lambda_a$-polarized K3 surface $Z_a$.
\end{proof}
\begin{remark} \label{slightgen-nodal2}
In the notation of the above proof, we have $h^0(N_{\bar{f}_a}) \leq p(\bar{B}_a)$ from Lemma \ref{stablem}. It thus follows that 
$h^0(N_{f_a}) \leq p(B_a)$ by semicontinuity.
\end{remark}

\begin{lem} \label{aghzt}
Let $1 \leq d \leq 5$ be an integer and consider the rank five lattice $K_d$ with basis $\{ A,B,\Gamma_1,\Gamma_2, \Gamma_3 \}$ giving the intersection matrix
\[ \left( \begin{array}{ccccc}
-2 & 6 & 3 & 2 & d \\
6 & 0 & 0 & 0 & 0 \\
3 & 0 & -2 & 0 & 0\\
2 & 0 & 0 & -2 & 0 \\
d & 0 & 0 & 0 & -2
 \end{array} \right). \]
 Then $K_d$ is an even lattice of signature $(1,4)$. There exists a K3 surface $Y_{K_d}$ with
 $Pic(Y_{K_d}) \simeq K_d$, and such that the classes $\{ A,B,\Gamma_1,\Gamma_2, \Gamma_3 \}$ are all represented by nodal, reduced curves such that the nodal curve $A$ meets $\Gamma_1, \Gamma_2$ and $\Gamma_3$ transversally. Further, $A+B$ is big and nef.
\end{lem}
\begin{proof}
Let $Y_1, Y_2$ be smooth elliptic curves and consider the Kummer surface $\widetilde{Z}$ associated to $Y_1 \times Y_2$. Let $P_1, P_2, P_3, P_4$ be the four $2$-torsion points of $Y_1$ and let $Q_1, Q_2, Q_3, Q_4$ be the $2$-torsion points of $Y_2$. Let $E_{i,j} \seq \widetilde{Z}$ denote the exceptional divisor over $P_i \times Q_j$, let $T_i \seq \widetilde{Z}$ denote the strict transform of $(P_i \times Y_2) / \pm$ and let
$S_j$ denote the strict transform of $(Y_1 \times Q_j) / \pm$. We also denote by $F$ a smooth elliptic curve of the form $x \times Y_2$, where $x \in Y_1$ is a non-torsion point. It may help the reader to consult the diagram on \cite[p.\ 344]{mori-mukai}, to see the configuration of these curves. We set 
\begin{align*}
\widetilde{A} &:= S_1+E_{1,1}+T_1+E_{1,2}+S_2+E_{1,3}+S_3 \\ 
\widetilde{B} &:=F+S_4+E_{2,4}+T_2+E_{2,1}+E_{2,2}+E_{2,3} \\ 
\widetilde{\Gamma}_1 &:=T_3+E_{3,1}+E_{3,2}+E_{3,3}\\
\widetilde{\Gamma}_2&:=T_2+E_{2,1}+E_{2,2} \\ 
\widetilde{\Gamma}_3 &= \left\{ \begin{array}{ll}
         T_4+\sum_{i=1}^d E_{4,i}, & \mbox{if $1 \leq d \leq 3$}\\
         T_4+\sum_{i=1}^{d-3} E_{4,i}+E_{4,4}+S_4+E_{2,4}+T_2+E_{2,1}+E_{2,2}+E_{2,3}, & \mbox{if $4 \leq d \leq 5$}.\end{array} \right.
\end{align*}
Then $\widetilde{A}$, $\widetilde{B}$, $\widetilde{\Gamma}_1, \widetilde{\Gamma}_2, \widetilde{\Gamma}_3$ generate $K_d $ (to simplify the computations, use that a tree of $-2$ curves has self-intersection $-2$). To see that this gives a primitive embedding of $K_d $ in $Pic(\widetilde{Z})$ we compute the intersections with elements of $Pic(\widetilde{Z})$; for $J \in Pic(\widetilde{Z})$, define $(J \cdot K_d)$ to be the quintuple
$((J \cdot \widetilde{A}), (J \cdot \widetilde{B}), (J \cdot \widetilde{\Gamma}_1), (J \cdot \widetilde{\Gamma}_2),(J \cdot \widetilde{\Gamma}_3))$. Then one computes
$(T_1 \cdot K_d)=(1,0,0,0,0)$, $(T_3 \cdot K_d)=(0,0,1,0,0)$, $(E_{4,1} \cdot K_d)=(1,0,0,0,-1)$,
$(E_{2,4} \cdot K_d)=(0,0,0,1,0)$, $(E_{2,3} \cdot K_d)=(1,-1,0,1, c)$, where $c$ is either $0$ or $-1$, depending on $d$. Thus $K_d$ is primitively embedded in $Pic(\widetilde{Z})$. Further, all intersections of  $\widetilde{B}$, $\widetilde{\Gamma}_1, \widetilde{\Gamma}_2, \widetilde{\Gamma}_3$ with $\widetilde{A}$ are transversal . Note that for any (rational) component $R \seq \widetilde{A}+\widetilde{B}$, $(R \cdot \widetilde{A}+\widetilde{B}) \geq 0$. Thus $\widetilde{A}+\widetilde{B}$ is big and nef. Hence the claim holds by degenerating to $\widetilde{Z}$.
\end{proof}

\begin{lem} \label{aghzt-2}
Let $1 \leq d \leq 5$ be an integer and consider the K3 surface $Y_{K_d}$ from Lemma \ref{aghzt}. Then the classes $\{ A,B,\Gamma_1,\Gamma_2, \Gamma_3 \}$ are all represented by integral curves. 
\end{lem}
\begin{proof}
We will firstly show that $B$ is nef, and hence base point free. Indeed, there would otherwise exist
an effective divisor $R=xA+yB+z\Gamma_1+w\Gamma_2+u\Gamma_3$ for integers $x,y,z,w,u$, with $(R)^2=-2$ and $(R \cdot B) < 0$, i.e.\ $x<0$. Thus $R-xA$ is effective and $(R-xA \cdot A)=(R-xA \cdot A+B) \geq 0$,
since $A+B$ is nef. But then
\begin{align*}
-2=(R)^2 &= ((R-xA)+xA)^2 \\
&=-2(z^2+w^2+u^2+x^2)+2x(R-xA \cdot A)
\end{align*}
So we must have $z=w=u=0$, $x=-1$ and $(R+A \cdot A)=0$. But then $R=-A+yB$ and $(R+A \cdot A)=0$ gives $y=0$. But this contradicts that $A$ is effective.

We next show that each $\Gamma_i$ is integral. Let $R$ be any irreducible component of $\Gamma_i$ with $(R \cdot \Gamma_i) <0$, $(R)^2=-2$ (such a component exists). Write $R=xA+yB+z\Gamma_1+w\Gamma_2+u\Gamma_3$. We have $x \geq 0$ as $(R \cdot B) \geq 0$. Assume $x \neq 0$. Then
$(R \cdot R+B)>0$, $(R +B)^2 >0$, so that $R+B$ is big and nef, contradicting that $(R+B \cdot \Gamma_i)=(R \cdot \Gamma_i) <0$. Thus $x=0$ and $R=yB+z\Gamma_1+w\Gamma_2+u \Gamma_3$. Since $(R)^2=-2$, $(R \cdot \Gamma_i) <0$, we have $R=yB+\Gamma_i$. Since $(R \cdot A+B)=(R \cdot A) \geq 0$, we must have $y \geq 0$ (for $i=3$, we need here that $d <6$). Since the only effective divisor in $|R|$ is integral (and equal to $R$), we must have $y=0$. Thus $\Gamma_i=R$ is integral. 

To show that $A$ is integral, let $R_1, \ldots, R_s$ be the components of the effective $-2$ curve $A$, and write $R_i=x_iA+y_iB+z_i\Gamma_1+w_i\Gamma_2+u_i\Gamma_3$ for integers $x_i,y_i,z_i,w_i,u_i$. Intersecting with $B$ shows we have $x_i \geq 0$ for all $i$. Thus there is precisely one component, say $R_1$ with $x_1 \neq 0$ and further $x_1=1$. Now, choose a component $R_i$ with $(R_i)^2=-2$, $(R_i \cdot A)<0$. Firstly assume $i \neq 1$, so that $x_i=0$. Intersecting with the integral curves $\Gamma_j$, $1 \leq j \leq 3$ (and noting $R_i \neq \Gamma_j$ as $(R_i \cdot A) <0$), we have $z_i, w_i, u_i \leq 0$. From $(R_i)^2=-2$, we see $R_i=y_iB-\Gamma_j$ for some $1 \leq j \leq 3$. Intersecting with $A$ gives $6y_i-k <0$ for some $1 \leq k \leq 5$, and thus $y_i \leq 0$ which is a contradiction to the effectivity of $R_i$.

In the second case, assume $(R_1)^2=-2$, $(R_1 \cdot A)<0$, with $R_1=A+y_1B+z_1\Gamma_1+w_1\Gamma_2+u_1\Gamma_3$. We compute
\begin{align*}
-2 &= (R_1)^2\\
&= ((R_1-A)+A)^2 \\
&=-2(y_1^2+z_1^2+w_1^2+u_1^2+1)+2((R_1 \cdot A)+2) \\
&= -2(y_1^2+z_1^2+w_1^2+u_1^2)+2((R_1 \cdot A)+1).
\end{align*} 
Thus we have either $(R_1 \cdot A)=-2$ and $R_1=A+y_1B$ or $(R \cdot A)=-1$, $R_1=A+y_1B \pm \Gamma_j$ for some $1\leq j \leq 3$. In the first case, $(A+y_1B \cdot A)=-2$ implies $y_1=0$ so $A=R_1$ is integral. In the second case, $(A+y_1B \pm \Gamma_j \cdot A)=-1$ implies $-1=-2+6y_1 \pm k$, for $1 \leq k=(A \cdot \Gamma_j) \leq 5$. The only possibilities are $y_1=0$, $k=1$, $R_1=A+\Gamma_j$, contradicting that all effective divisors in $|R_1|$ are integral, or $y_1=1$, $k=5$, $R_1=A+B-\Gamma_j$. Since $(B-\Gamma_i)^2=-2$, $(B-\Gamma_j \cdot A+B)>0$, we have that $B-\Gamma_j$ is effective, so once again this contradicts that all effective divisors in $|R_1|$ are integral.
\end{proof}

\begin{lem} \label{essential-def}
Let $M_{\Lambda_a}$ denote the moduli space of pseudo-ample $\Lambda_a$-polarized K3 surfaces, with the lattice $\Lambda_a$ as in Lemma \ref{hjz}, for $14 \leq a \leq 19$. Then there is a nonempty open subset $U \seq M_{\Lambda_a}$ such that for $[Y_a] \in U$  we may pick the basis $\{ D, F, \Gamma \}$ such that 
there is an integral, nodal, rational curve $R_a \in |D-2F-\Gamma|$ such that $R_a$ meets $\Gamma$ transversally in three points.
\end{lem}
\begin{proof}
Set $y:=a-14$, so that $0 \leq y \leq 5$ by assumption.
If we change the basis of $\Lambda_a$ to $\{ D-2F-\Gamma, F, \Gamma \}=\{X,Y,Z \}$, we see $\Lambda_a$ is isometric to the lattice 
$\bar{\Lambda}_a$ with intersection matrix
\[
\left( \begin{array}{ccc}
2y-2 & 6 & 3 \\
6 & 0 & 0 \\
3 & 0 & -2 \end{array} \right).\] Let $K_d$ be the lattice from Lemma \ref{aghzt}.
For appropriate choices of $d$ there is a primitive lattice embedding
$\bar{\Lambda}_a \hookrightarrow K_d$, given by
\begin{align*}
X &\mapsto  A+\epsilon_1\Gamma_3+\epsilon_2 \Gamma_2 \\
Y &\mapsto B \\
Z &\mapsto \Gamma_1
\end{align*}
where if $a =14$ we set $\epsilon_1=\epsilon_2=0$ and $d$ arbitrary, if $15 \leq a < 19$ we set $d=y+1$, $\epsilon_1=1$, $\epsilon_2=0$ and if $a=19$ we set $d=5$ and $\epsilon_1=\epsilon_2=1$. By Lemmas \ref{aghzt}, \ref{aghzt-2}, we have that $A,\Gamma_2, \Gamma_3 $ are represented by smooth rational curves intersecting transversally on $Y_{K_d}$. In all cases, the divisor $D=X+2Y+Z$ is big and nef, since if $R \in \{ \Gamma_1, \epsilon_1 \Gamma_3, \epsilon_2 \Gamma_2, A \}$, $(R \cdot D) \geq 0$. By partially normalising nodes we may produce an unramified, stable map $f: \tilde{C} \to Y_{K_d}$, where $\tilde{C}$ is a genus zero union of smooth rational curves and $f(\tilde{C})=A+\epsilon_1 \Gamma_3+\epsilon_2 \Gamma_2$. By \cite[\S 2, Remark 3.1]{kool-thomas} we may deform $f$ horizontally to a stable map with target a small deformation of $Y_{K_d}$ in the moduli space of $\Lambda_a$-polarised K3 surfaces.
After deforming $Y_{K_d}$ to a K3 surface $Y_a$ with $Pic(Y_a) \simeq \Lambda_a$, we can deform $A+\epsilon_1 \Gamma_3+\epsilon_2 \Gamma_2$ to a nodal rational curve $\bar{R}_a$ which meets $\Gamma$ transversally in three points. Furthermore, $\bar{R}_a$ is integral, since $\Lambda_a$ contains no $-2$ curves which have zero intersection with $F, \Gamma$. 
\end{proof}

We now prove Theorem \ref{finiteness-nonprim}:
\begin{proof} [Proof of Theorem \ref{finiteness-nonprim}]
Assume $k \geq 2$, $g \geq 8$. Set $m:= \left \lfloor \frac{g-5}{6} \right \rfloor$ and let $0 \leq r(g) \leq 5$ be the unique integer such that
$ g-5=6m+r(g)$. We let $l_g:=15$ if $r(g)=3,4$, $m$ odd and/or $k$ even, $l_g:=16$ if $r(g)=3,4$, $m$ even and $k$ odd, $l_g:=17$ if $r(g)=5$, $m$ odd and/or $k$ even, $l_g:=18$ if $r(g)=5$, $m$ even and $k$ odd, $l_g:=17$, if $r(g)\leq 2$, $m$ even and/or $k$ even and finally
$l_g:=18$ if $r(g) \leq 2$, $m$ odd and $k$ odd. We will prove that there is a component $I \seq \mathcal{T}^n_{g,k}$ such that $${\eta}_{|_I}: I \to \mathcal{M}_{p(g,k)-n}$$ is generically finite for
$p(g,k)-n \geq l_g$. Further, we show that for the general $[\tilde{f}: C' \to X'] \in I$, $C'$ is non-trigonal.

Consider the $\Lambda_a$-polarized K3 surface $Y_a$ from Lemma \ref{essential-def} and let $\{ D, F, \Gamma \}$ be as in the lemma. Set $m:= \left \lfloor \frac{g-5}{6} \right \rfloor \geq 0$ and $$m'= \begin{cases} m-1 &\mbox{if } r(g) \geq 3 \\ 
m-2 & \mbox{if } r(g) \leq 2. \end{cases}. 
$$ Consider the primitive, ample line bundle $H=D+m'F $. We choose
$$a= \begin{cases} 11+r(g) &\mbox{if } r(g) \geq 3 \\ 
17+r(g) & \mbox{if } r(g) \leq 2. \end{cases}. 
$$
Then $(H)^2=2g-2$ for $g \geq 8$. Let $f_a: B_a \to Y_a$ (resp.\ $R_a$) be the unramified stable map (resp.\ rational curve) from  lemmas \ref{hjz} (resp.\ \ref{essential-def}). Set $l=k m'+2(k-1)$ which is nonnegative for $k \geq 2$, $g \geq 8$. We have an effective decomposition
$$kH \sim f_a(B_a)+(k-1) R_a+(k-1)\Gamma +lF_0 ,$$
where $F_0 \in |F|$ is an integral, nodal rational curve as in lemma \ref{hjz}.  We will prove the result by constructing an unramified stable map $f: B \to Y_a$ with $f_* (B)= f_a(B_a)+(k-1)R_a+(k-1)\Gamma +lF_0$ satisfying the conditions of Proposition \ref{finiteness-criterion}.

Assume firstly $m'$ is even and/or $k$ is even, so $l$ is even, and set $s=l/ 2$. Let $\proj^1 \to F_0$ be the normalization morphism and let $p,q$ be the points over the node. Let $x$ be the point of intersection of $f_a(B_a)$ and $\Gamma$ and let $y,z$ be distinct points in $\Gamma \cap R_a$. We may pick the points to ensure $y \neq x, z \neq x$. Define $B$ as the union of $B_a$ with $l+2(k-1)$ copies of $\proj^1$ and with transversal intersections as in the following diagram:
$$
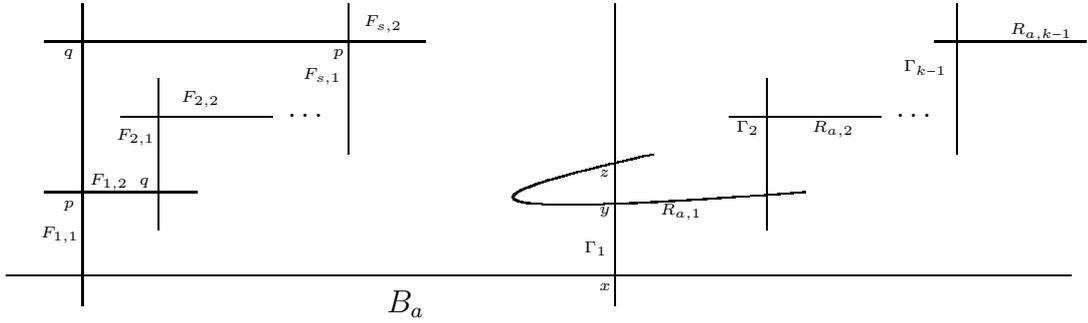
\setlength{\unitlength}{1cm}
\begin{picture}(14,5)
\put(1,0){\line(0,1){4}}
\put(0.5,1.5){\line(1,0){2}}
\put(2,1){\line(0,1){2}}
\put(1.5,2.5){\line(1,0){2}}
\put(3.7,2.5){$\ldots$}
\put(4.5,2){\line(0,1){2}}
\put(0.5,3.5){\line(1,0){5}}

\put(8,0){\line(0,1){4}}
\qbezier(8.5,2)(4,1)(10.5,1.5)
\put(10,1){\line(0,1){2}}
\put(9.5,2.5){\line(1,0){2}}
\put(11.7,2.5){$\ldots$}
\put(12.5,2){\line(0,1){2}}
\put(12.2,3.5){\line(1,0){2}}

\put(0,0.4){\line(1,0){14}}

\put(5,-0.1){\mbox{$B_a$}}
\put(0.45,0.9){\tiny\mbox{$F_{1,1}$}}
\put(0.75,1.3){\tiny\mbox{$p$}}
\put(0.75,3.3){\tiny\mbox{$q$}}
\put(1.1,1.6){\tiny\mbox{$F_{1,2}$}}
\put(1.75,1.6){\tiny\mbox{$q$}}
\put(1.45,2.2){\tiny\mbox{$F_{2,1}$}}
\put(2.3,2.7){\tiny\mbox{$F_{2,2}$}}
\put(3.9,3){\tiny\mbox{$F_{s,1}$}}
\put(4.3,3.3){\tiny\mbox{$p$}}
\put(4.7,3.7){\tiny\mbox{$F_{s,2}$}}

\put(7.8,0.2){\tiny\mbox{$x$}}
\put(7.6,0.7){\tiny\mbox{$\Gamma_1$}}
\put(8.6,1.2){\tiny\mbox{$R_{a,1}$}}
\put(9.6,2.3){\tiny\mbox{$\Gamma_2$}}
\put(10.6,2.3){\tiny\mbox{$R_{a,2}$}}
\put(11.8,3.1){\tiny\mbox{$\Gamma_{k-1}$}}
\put(13.2,3.6){\tiny\mbox{$R_{a,k-1}$}}
\put(7.8,1.22){\tiny\mbox{$y$}}
\put(7.8,1.7){\tiny\mbox{$z$}}
\end{picture}
$$
Then there is a unique unramified morphism $f: B \to Y_a$ with $f_* (B)= f_a(B_a)+(k-1)R_a+(k-1)\Gamma +lF_0$ which restricts to the normalization $\proj^1 \to R_a$ on all components marked $R_{a,i}$, restricts to $f_a$ on $B_a$, sends all components marked $F_{i,j}$ to $F_0$, all components marked $\Gamma_i$ to $\Gamma$ and which takes points marked $x$ (resp.\ $y,z,p,q$) to $x$ (resp.\ $y,z,p,q$). 

We now claim that if $B_0 \seq B$ is a connected union of components containing $R_{a,k-1}$ with $f_*(B_0) \in |nH|$ then $n=k$ and $B_0 = B$. If $c_1D+c_2F+c_3\Gamma$ is a divisor linearly equivalent to $nH$, then intersecting with $F$ shows $c_1=n$. Now $R_{a} \in |D-2F-\Gamma|$, whereas $H =D+(m-1)F$ for $m \geq 0$. Thus $f_*(B_0) \in |nH|$ shows that the connected curve $B_0$ cannot coincide with the component $R_{a,k-1}$, and thus must also contain $\Gamma_{k-1}$. Repeating this argument, one sees readily that $B_0$ must contain $\sum_{i=1}^{k-1} (\Gamma_i+R_{a,i})+B_a$. We then get $n=k$ as required, which forces $B_0=B$. 

Using Remarks \ref{slightgen-nodal} and \ref{slightgen-nodal2}, one sees $h^0(N_{f}) \leq p(B)$. For any component $C \neq B_a \seq B$, $f(C)$ meets $f(B_a)$ properly. Thus it follows from Proposition \ref{prim-cor} that the conditions of Proposition \ref{finiteness-criterion} are met. Note that the arithmetic genus of $B$ is $l_g$. 

Now assume $m'$ is odd and $k$ is odd. Let $a,b,c \in f_a(B_a) \cap F_0$ be distinct points. Let $B$ be as in the diagram below.
$$
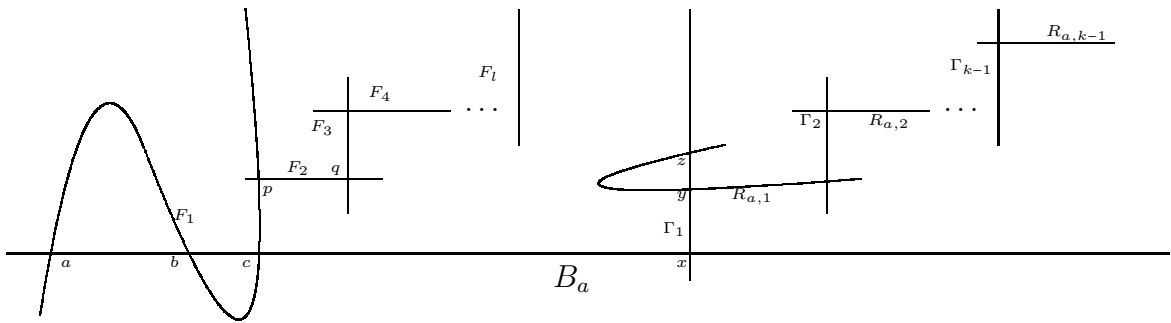
\setlength{\unitlength}{0.9cm}
\begin{picture}(14,5)
\qbezier(0,2)(2.3,-4)(1.5,4)
\qbezier(-1.5,-0.5)(-0.8,4)(0,2)
\put(1.5,1.5){\line(1,0){2}}
\put(3,1){\line(0,1){2}}
\put(2.5,2.5){\line(1,0){2}}
\put(4.7,2.5){$\ldots$}
\put(5.5,2){\line(0,1){2}}
%\put(4.3,3.5){\line(1,0){1.5}}

\put(8,0){\line(0,1){4}}
\qbezier(8.5,2)(4,1)(10.5,1.5)
\put(10,1){\line(0,1){2}}
\put(9.5,2.5){\line(1,0){2}}
\put(11.7,2.5){$\ldots$}
\put(12.5,2){\line(0,1){2}}
\put(12.2,3.5){\line(1,0){2}}

\put(-2,0.4){\line(1,0){17}}

\put(6,-0.1){\mbox{$B_a$}}
\put(0.45,0.9){\tiny\mbox{$F_{1}$}}
\put(1.75,1.3){\tiny\mbox{$p$}}
\put(2.1,1.6){\tiny\mbox{$F_{2}$}}
\put(2.75,1.6){\tiny\mbox{$q$}}
\put(2.45,2.2){\tiny\mbox{$F_{3}$}}
\put(3.3,2.7){\tiny\mbox{$F_{4}$}}
\put(4.9,3){\tiny\mbox{$F_{l}$}}
%\put(5.3,3.3){\tiny\mbox{$p$}}
%\put(5.7,3.7){\tiny\mbox{$F_{s,2}$}}

\put(-1.2,0.2){\tiny\mbox{$a$}}
\put(0.4,0.2){\tiny\mbox{$b$}}
\put(1.45,0.2){\tiny\mbox{$c$}}

\put(7.8,0.2){\tiny\mbox{$x$}}
\put(7.6,0.7){\tiny\mbox{$\Gamma_1$}}
\put(8.6,1.2){\tiny\mbox{$R_{a,1}$}}
\put(9.6,2.3){\tiny\mbox{$\Gamma_2$}}
\put(10.6,2.3){\tiny\mbox{$R_{a,2}$}}
\put(11.8,3.1){\tiny\mbox{$\Gamma_{k-1}$}}
\put(13.2,3.6){\tiny\mbox{$R_{a,k-1}$}}
\put(7.8,1.22){\tiny\mbox{$y$}}
\put(7.8,1.7){\tiny\mbox{$z$}}
\end{picture}
$$
\\
Then as before there is an unramified morphism $f: B \to Y_a$ with  $f_* (B)= f_a(B_a)+(k-1)R_a+(k-1)\Gamma +lF_0$ satisfying the conditions of Proposition \ref{finiteness-criterion}, and $B$ has arithmetic genus $l_g$. 
\end{proof}

The following lemma will be needed for Theorem \ref{marked-wahl-k3}.
\begin{lem} \label{coh-van}
Assume $p(g,k)$ and $n$ are such that there is a component $I \seq \mathcal{T}^n_{g,k}$ such that the morphism
${\eta}_{|_I}: I \to \mathcal{M}_{p(g,k)-n}$ is generically finite. Then for the general $[(f: B \to X,L)] \in I$, we have
$$H^0(B, f^*(T_X))=0.$$
\end{lem}
\begin{proof}
Let $T^n_{g,k}(X,L)$ denote the fibre of $\mathcal{T}^n_{g,k} \to \mathcal{B}_g$ over $[(X,L)]$. 
The finiteness of ${\eta}_{|_I}$ at $[(f: B \to X,L)] $ obviously implies that the morphism
\begin{align*}
r_{n,k} \;: \;T^n_{g,k}(X,L) & \to \mathcal{M}_{p(g,k)-n} \\
[f: B \to X] & \mapsto [B]
\end{align*}
is finite near $[f]$. The claim $H^0(B, f^*(T_X))=0$ then follows from the exact sequence of sheaves on $B$
$$ 0 \to T_{B} \to f^*(T_X) \to N_f \to 0$$
and the fact that the coboundary morphism $H^0(B, N_f) \to H^1(B, T_{B})$ corresponds to the differential of $r_{n,k}$.
\end{proof}
\begin{rem}
For $g \geq 51$, $n \leq \frac{g-50}{2}$ then one has $H^0(B, f^*(T_X))=0$ for the general $[(f: B \to X,L)] \in I$ in \emph{every} component of $\mathcal{V}^n_g \seq \mathcal{T}^n_{g}$, by \cite{halic}. Similar bounds for the case $k \geq 2$ are also stated. In this paper, generic finiteness results for $\mathcal{V}^{n}_{g,k} \to \mathcal{M}_{p(g,k)-n}$ were also claimed, the proof however seems flawed, see Remark \ref{halic-remarks}.
\end{rem}

\section{The marked Wahl map} \label{markywahl}
 Recall the following definition from \cite{wahl-curve}: let $V$ be any smooth projective variety, and let $R$ be a line bundle on $V$. Then there is a linear map, called the \emph{Gaussian}:
\begin{align*}
\Phi_R \; : \; \bigwedge^2 H^0(V,R) & \to H^0(V, \Omega_V (R^2)) \\
 s \wedge t & \mapsto sdt-tds.
\end{align*} 
Here $\Omega_V (R^2)$ denotes $\Omega_V \otimes R^2$. In the case $R = \omega_V$, this map is called the \emph{Wahl map}. For $V=C$ a smooth curve, and $T \seq C$ a marking, we call $\Phi_{\omega_C(-T)}$ the \emph{marked Wahl map}.  In this section we will use an approach inspired by \cite{wahl-plane-nodal} to study Gaussians in the case where $V$ is a general curve and $R$ is a twist of the canonical bundle.

We begin with the following lemma, which is a special case of \cite[Lem.\ 3.3.1]{greuel}:
\begin{lem} \label{jjji}
Let $x_1, \ldots, x_n, y_1, \ldots, y_m$ be distinct, generic points of $\proj^2$ and let $d$ be a positive integer satisfying
$$3n+6m < \frac{d^2+6d-1}{4}-\left \lfloor \frac{d}{2} \right \rfloor .$$
Then there exists an integral curve $C \seq \proj^2$ of degree $d$ with nodes at $x_i$, ordinary singular points of multiplicity $3$ at $y_j$ for $1 \leq i \leq n$, $1 \leq j \leq m$ and no other singularities.
\end{lem}
Let $C \seq \proj^2$  be an integral curve of degree $d$ with nodes at $x_i$, ordinary singular points of multiplicity three at $y_j$ for $1 \leq i \leq n$, $1 \leq j \leq m$ and no other singularities, as in the lemma above. Let $\pi: S \to \proj^2$ be the blow-up at $x_1, \ldots, x_n, y_1, \ldots, y_m$, let $E_X$ be the sum of the exceptional divisors over $x_i$ for $1 \leq i \leq n$ and let $E_Y$ be the sum of the exceptional divisors over $y_j$ for $1 \leq j \leq m$. Denote by $D$ the strict transform of $C$, and let $T \seq D$ be the marking $E_X \cap D$. Note that $D$ is smooth, since all singularities are ordinary. Set $$M=\mathcal{O}_S((d-3)H-2E_X-2E_Y),$$ where $H$ denotes the pull-back of the hyperplane of $\proj^2$. Note that 
$$K_D \sim (D+K_S)_{|_D} \sim (dH-2E_X-3E_Y)_{|_D}+(E_X+E_Y-3H)_{|_D},$$ and this gives
$$ M_{|_D}\simeq K_D(-T).$$  We therefore have the following commutative diagram
\begin{align} \label{gauss-diag}
\xymatrix{
\bigwedge^2 H^0(S, M) \ar[r]^{\; \; \; \; \Phi_{M} \; \; \; \;}  \ar[d] & H^0(S, \Omega_{S}(M^2)) \ar[d]^g \\
\bigwedge^2 H^0(D, K_{D}(-T)) \ar[r]^{\; \; \; \; \; W_{D,T} \; \; \; \; \; \; }  &  H^0(D, K^3_{D}(-2T)).
}
\end{align}
where $\Phi_{M}$ is the Gaussian, \cite[\S 1]{wahl-curve} and $ W_{D,T} $ is the marked Wahl map of $(D,T)$. Here $g$ denotes the composition of the natural maps
$H^0(S, \Omega_{S}(M^2)) \to H^0(D,{\Omega_{S}(M^2)}_{|_D})$ and $H^0(D,{\Omega_{S}(M^2)}_{|_D}) \to H^0(D,\Omega_{D}(M^2))$. We aim to show that $W_{D,T}$ is surjective. We will firstly show that $g$ is surjective. The main tool we will need is the Hirschowitz criterion, \cite{hirschowitz}: 
\begin{thm}[Hirschowitz] \label{hirsch}
Let $p_1, \ldots, p_t$ be generic points in the plane, and assume $m_1, \ldots, m_t, d$ are nonnegative integers satisfying
$$ \sum_{i=1}^t \frac{m_i(m_i+1)}{2} < \left \lfloor \frac{(d+3)^2}{4} \right \rfloor .$$
Then $H^1(\proj^2, \mathcal{O}_{\proj^2}(d) \otimes I_{p_1}^{m_1} \otimes \ldots \otimes I_{p_t}^{m_t})=0$.
\end{thm}
\begin{lem} \label{gauss-lem1}
Assume $d \geq 8$, $3n+m < \left \lfloor \frac{(d-5)^2}{4} \right \rfloor$. Then $H^1(S,\Omega_S((d-6)H-2E_X-E_Y))=0.$
\end{lem}
\begin{proof}
The relative cotangent sequence twisted by $ (d-6)H-2E_X-E_Y$ gives
$$0 \to \pi^* \Omega_{\proj^2}((d-6)H-2E_X-E_Y) \to \Omega_{S}((d-6)H-2E_X-E_Y) \to \omega_{E_X}(2) \oplus \omega_{E_Y}(1) \to 0 .$$ Thus it suffices to show
$$H^1(\proj^2, \Omega_{\proj^2}(d-6) \otimes I_X^2 \otimes I_Y )=H^1(S, \pi^* \Omega_{\proj^2}((d-6)H-2E_X-E_Y) )=0$$ 
where $I_X=I_{x_1} \otimes \ldots \otimes I_{x_n}$ and $I_Y=I_{y_1} \otimes \ldots \otimes I_{y_m}$.
Twisting the Euler sequence by $\omega_{\proj^2}$ gives a short exact sequence
$$0 \to \omega_{\proj^2} \to \mathcal{O}_{\proj^2}(-2)^{\oplus 3} \to \Omega_{\proj^2} \to 0 ,$$
where we have used the standard identification
$$\Omega_{\proj^2} \simeq T_{\proj^2} \otimes \omega_{\proj^2},$$
from \cite[Ex.\ II.5.16(b)]{har}. As $H^2(\proj^2, \omega_{\proj^2}(d-6) \otimes I_X^2 \otimes I_Y)=0$ for $d \geq 7$, it suffices to show $H^1(\proj^2,  \mathcal{O}_{\proj^2}(d-8)\otimes I_X^2 \otimes I_Y )=0$. This follows from Theorem \ref{hirsch} and the assumption $d \geq 8$, $3n+m < \left \lfloor \frac{(d-5)^2}{4} \right \rfloor$.
\end{proof}
\begin{lem} \label{gauss-lem2}
Let $D$ be as above and assume $3n+m < \left \lfloor \frac{(d-3)^2}{4} \right \rfloor$. Assume further that $m \geq 10$, so that $d \geq 10$. Then $$H^1(D,\mathcal{O}_D((d-6)H-2E_X-E_Y))=0.$$
\end{lem}
\begin{proof}
We have an exact sequence
$$ 0 \to \mathcal{O}_S(-6H+2E_Y) \to \mathcal{O}_S((d-6)H-2E_X-E_Y) \to \mathcal{O}_D((d-6)H-2E_X-E_Y) \to 0.  $$
By the Hirschowitz criterion, $H^1(S,\mathcal{O}_S((d-6)H-2E_X-E_Y))=0$, as we are assuming $d \geq 6$, $3n+m < \left \lfloor \frac{(d-3)^2}{4} \right \rfloor$. Thus it suffices to show $H^2(S,\mathcal{O}_S(-6H+2E_Y) )=0$. By Serre duality, $h^2(S,\mathcal{O}_S(-6H+2E_Y) )=h^0(S,\mathcal{O}_S(3H+E_X-E_Y))$.
We have 
$$ 0 \to \mathcal{O}_S(3H-E_Y) \to \mathcal{O}_S(3H+E_X-E_Y) \to \mathcal{O}_{E_X}(-1) \to 0$$
and so it suffices to show $H^0(S,\mathcal{O}_S(3H-E_Y))=0$. But $H^0(S,\mathcal{O}_S(3H-E_Y))=H^0(\proj^2, \mathcal{O}(3) \otimes I_Y)=0$, since $m \geq 10$, and any ten general points do not lie on any plane cubic (as the space of plane cubics has dimension nine).
\end{proof}
\begin{cor} \label{f-surj}
Let $x_1, \ldots, x_n, y_1, \ldots, y_m$ be distinct, general points of $\proj^2$ with $m \geq 10$ and let $d \geq 21$ be a positive integer satisfying
$$3n+6m < \left \lfloor \frac{(d-5)^2}{4} \right \rfloor.$$ Let  $C \seq \proj^2$ be an integral curve as in Lemma \ref{jjji}. Let $S \to \proj^2$ denote the blow-up of $\proj^2$ at $x_1, \ldots, x_n, y_1, \ldots, y_m$, and let $D \seq S$ denote the strict transform of $C$. Then the map $g$ from Diagram (\ref{gauss-diag}) is surjective.
\end{cor}
\begin{proof}
Note that $ \left \lfloor \frac{(d-5)^2}{4} \right \rfloor < \frac{d^2+6d-1}{4}-\left \lfloor \frac{d}{2} \right \rfloor $ for $d \geq 5$ so that such a curve $C$ exists. Let $M$ be the line bundle defined above Diagram (\ref{gauss-diag}).
We have short exact sequences
$$ 0 \to \Omega_S((d-6)H-2E_X-E_Y) \to \Omega_S(M^2) \to \Omega_S(M^2)_{|_D} \to 0 $$ and
$$0 \to \mathcal{O}_D ((d-6)H-2E_X-E_Y) \to  \Omega_S(M^2)_{|_D} \to \Omega_D(M^2) \to 0 .$$ 
The map $f$ is the composition of the natural maps $H^0(S,\Omega_S(M^2)) \to H^0(S, \Omega_S(M^2)_{|_D})$ and 
$H^0(D,\Omega_S(M^2)_{|_D} ) \to H^0(D, \Omega_D(M^2) )$, so the claim follows from lemmas \ref{gauss-lem1} and \ref{gauss-lem2}.
\end{proof}

We now wish to show that the Gaussian $\Phi_M$ from Diagram \ref{gauss-diag} is surjective. We start by recalling one construction of Gaussian maps from \cite[\S 1]{wahl-curve}. Let $X$ be a smooth, projective variety, and $L \in Pic(X)$ a line bundle. Let $Y \to X \times X$ be the blow-up of the diagonal $\Delta$, and let $F$ denote the exceptional divisor. There is a short exact sequence of sheaves on $X \times X$
$$ 0 \to I_{\Delta}^2 \to I_{\Delta} \to \Delta_* \Omega_X \to 0.$$ Twisting the above sequence by $L \boxtimes L$ produces a short exact sequence
$$ 0 \to I_{\Delta}^2 (L \boxtimes L) \to I_{\Delta}(L \boxtimes L) \to \Delta_* (\Omega_X(L^2)) \to 0$$ and upon taking cohomology we get a map
$$ \widetilde{\Phi}_L : H^0(X \times X,I_{\Delta}(L \boxtimes L)) \to H^0(X, \Omega_X(L^2)).$$ Now $H^0(X \times X,I_{\Delta}(L \boxtimes L)) $ may be identified with the kernel $\mathcal{R}(L,L)$ of the
multiplication map $H^0(X,L) \otimes H^0(X,L) \to H^0(X,L^2)$, and we have $\bigwedge^2 H^0(X,L) \seq  \mathcal{R}(L,L)$ by sending $s \wedge t$ to $s \otimes t-t\otimes s$. Further, $\Phi_L$ is the restriction of $\widetilde{\Phi}_L$ to $\bigwedge^2 H^0(X,L)$, and it is easily verified that both $\Phi_L$ and $\widetilde{\Phi}_L$ have the same image in $H^0(X, \Omega_X(L^2))$. Thus, to verify the surjectivity of $\Phi_L$, it suffices to show 
$$H^1(X \times X,I_{\Delta}^2 (L \boxtimes L))=H^1(Y, L_1+L_2-2F)=0$$
where $L_1$ and $L_2$ denote the pull-backs of $L$ via the projections $pr_i: Y \to X \times X \to X$, for $i=1,2$.

Following \cite{cili-corank}, \cite{wahl-plane-nodal}, we now wish to use the Kawamata--Viehweg vanishing theorem to show $H^1(Y, L_1+L_2-2F)=0$.
\begin{prop}[\!\! \cite{cili-corank}] \label{kv-van}
Let $X$ be a smooth projective surface, which is not isomorphic to $\proj^2$. Assume $L \in Pic(X)$ is a line bundle such that there exist three very ample line bundles $M_1, M_2, M_3$ with $L-K_X \sim M_1+M_2+M_3$. Then the Gaussian $\Phi_L$ is surjective.
\end{prop}
\begin{proof}
For a line bundle $A$ on $X$, we denote by $A_i \in Pic(Y)$ the pullback via the projection $pr_i: Y \to X \times X \to X$, for $i=1,2$.
By the above discussion, it suffices to show $H^1(Y, L_1+L_2-2F)=0$. As the diagonal $\Delta \seq X \times X$ has codimension two, we have $K_Y \simeq g^*K_X+F$, \cite[Exercise II.8.5]{har}. Thus we see $H^1(Y,L_1+L_2-2F)=H^1(Y,(L-K_X)_1+(L-K_X)_2-3F+K_Y)$, and so by the Kawamata--Viehweg vanishing theorem it suffices to show
$(L-K_X)_1+(L-K_X)_2-3F$ is big and nef. Since we have
\begin{align*}
(L-K_X)_1+(L-K_X)_2-3F &= (M_{1,1}+M_{1,2}-F) + (M_{2,1}+M_{2,2}-F) \\
& + (M_{3,1}+M_{3,2}-F), 
\end{align*}
it suffices to show that $M_{i,1}+M_{i,2}-F$ is big and nef for $1 \leq i \leq 3$. Now $H^0(Y,M_{i,1}+M_{i,2}-F)$ is the kernel $\mathcal{R}(M_i,M_i)$ of the
multiplication map $H^0(X,M_i) \otimes H^0(X,M_i) \to H^0(X,M_i^2)$, and we have an injective map $\bigwedge^2 H^0(X,M_i) \hookrightarrow \mathcal{R}(M_i,M_i)$ 
sending $s \wedge t$ to $s \otimes t-t\otimes s$. Thus $\bigwedge^2 H^0(X,M_i) $ induces a sublinear system of $|M_{i,1}+M_{i,2}-F|$ which induces a rational map
\begin{align*}
\psi_i \; : \; Y &\to Gr(1, \proj(H^0(M_i)^*)) \\
(x,y) &\mapsto \overline{\phi_i(x) \phi_i(y)} 
\end{align*}
where $\phi_i: X \hookrightarrow \proj(H^0(M_i)^*)$ is the embedding induced by $M_i$, and where $\overline{\phi_i(x) \phi_i(y)} $ denotes the line through $\phi_i(x) $ and $\phi_i(y)$. By viewing $(x,y) \in F$ as a pair $x \in X$, $y \in T_{X,x}$, one sees that the map $\psi_i $ is in fact globally defined, and hence $M_{i,1}+M_{i,2}-F$ is nef. To see that it is big, it suffices to show that $\psi_i$ is generically finite, i.e.\ we need to show that there exist points $x,y \in X$ such that $\phi_i(X)$ does not contain the line $\overline{\phi_i(x) \phi_i(y)}$. But if this were not the case $\phi_i(X)$ would be a linear space, contrary to the hypotheses.
\end{proof}
We now return to the situation of the blown-up plane. We start with the following:
\begin{thm}[\!\! \cite{hirsch-zeit}] \label{hirsch-almeida}
Let $p_1, \ldots, p_k$ be generic distinct points in the plane, and let $\pi: S \to \proj^2$ be the blow-up. Let $E \seq S$ be the exceptional divisor, and let $H$ be the pull-back of the hyperplane class on $\proj^2$. If we assume
$d \geq 5$ and $k+6 \leq \frac{(d+1)(d+2)}{2},$
then $dH-E$ is very ample on $S$.
\end{thm}
Putting everything together, we deduce:
\begin{prop} \label{mark-surj-main}
Let $x_1, \ldots, x_n, y_1, \ldots, y_m$ be distinct, generic points of $\proj^2$ with $m \geq 10$ and let $d \geq 24$ be a sufficiently large integer, so that both of the following two conditions are satisfied
\begin{enumerate}
\item $3n+6m < \left \lfloor \frac{(d-5)^2}{4} \right \rfloor$,
\item $n+m+6 \leq \frac{(d-6)(d-3)}{18}$.
\end{enumerate}
Let  $C \seq \proj^2$ be an integral curve of degree $d$ with nodes at $x_i$, ordinary singular points of multiplicity $3$ at $y_j$ for $1 \leq i \leq n$, $1 \leq j \leq m$ and no other singularities. Let $S \to \proj^2$ denote the blow-up of $\proj^2$, and let $D \seq S$ denote the strict transform of $C$. Then the marked Wahl map $W_{D,T}$ is surjective, where $T$ is the divisor over the nodes of $C$. Furthermore, $h^0(D,\mathcal{O}_D(T))=1$.
\end{prop}
\begin{proof}
We will firstly show that $W_{D,T}$ is surjective. We have already seen in Corollary \ref{f-surj} that the map $f$ from Diagram \ref{gauss-diag} is surjective, thanks to the assumption $3n+6m < \left \lfloor \frac{(d-5)^2}{4} \right \rfloor$. Thus it suffices to show that $\Phi_M$ is surjective, where
$M=\mathcal{O}_S((d-3)H-2E_X-2E_Y)$. Now $M-K_S \sim (d-6)H-3E_X-3E_Y$ can be written as the sum
\begin{align*}
M-K_S &= (\lfloor \frac{d-6}{3}\rfloor H-E_X-E_Y)+ (\lfloor \frac{d-6}{3}\rfloor H-E_X-E_Y)\\
&+ ((d-6-2\lfloor \frac{d-6}{3}\rfloor)H-E_X-E_Y).
\end{align*}
Since we are assuming $d \geq 24$, $n+m+6 \leq \frac{(d-6)(d-3)}{18}$, Theorem \ref{hirsch-almeida} shows that $M-K_S$ may be written as a sum of three very ample line bundles
(use $\lfloor \frac{d-6}{3}\rfloor \geq \frac{d-9}{3}$). Thus Proposition \ref{kv-van} implies that the Gaussian $\Phi_M$ is surjective.

For the second statement, note that the short exact sequence
$$0 \to \mathcal{O}_S \to \mathcal{O}_S(E_X) \to \mathcal{O}_{E_X}(-1) \to 0 $$
gives that $h^0(S, \mathcal{O}_S(E_X) )=1$. From the sequence
$$ 0 \to \mathcal{O}_S(3E_X+3E_Y-dH) \to \mathcal{O}_S(E_X) \to \mathcal{O}_D(T) \to 0$$
it suffices to show $H^1(S,\mathcal{O}_S(3E_X+3E_Y-dH))=0$. But $dH-3E_X-3E_Y$ is a sum of three very ample line bundles, so it is big and nef, so that this follows from the Kawamata--Viehweg vanishing theorem and Serre duality.

\end{proof}
As an immediate consequence we can now show that, for any integer $l \in \mathbb{Z}$, there are infinitely many integers $h(l)$ such that the general marked curve $[(C,T)] \in \widetilde{\mathcal{M}}_{h(l),2l}$ has surjective marked Wahl map.
\begin{proof} [Proof of Theorem \ref{inf-many-wahl}]
Consider the curve $D \seq S$ from Proposition \ref{mark-surj-main}, applied to $n=l$ and $m=10$ (choose any $d$ satisfying the hypotheses of the proposition). Let $h(l)$ denote the genus of $D$. In an open subset about  $[(D,T)] \in \widetilde{\mathcal{M}}_{h(l),2l}$, we have $h^0(D,\mathcal{O}_D(T))=1$ and thus 
$h^0(D,K_D(-T))=\chi(K_D(-T))+1$ is locally constant. Further, the equality $h^0(D,K_D^3(-2T))=\chi(K_D^3(-2T))$ holds, since $\deg(K_D^3(-2T)) >2h(l)-2$. Thus the claim follows immediately from Proposition \ref{mark-surj-main} and semicontinuity.
\end{proof}
\begin{remark}
In our example $(D,T)$, we have that $l$ is of the order $\frac{h}{8}$, where $h$ is the genus of $D$. Indeed, if $n=l$ is large, we can take $d^2$ to be approximately
$18l$, so that $g(C)=\frac{1}{2}(d-1)(d-2) \sim 9l$ and $h \sim 8l$. Thus one would expect that the marked Wahl map of a general marked curve in $\widetilde{\mathcal{M}}_{h,2l}$ is surjective, so long as $l$ is at most of the order $\frac{h}{8}$.
\end{remark}

We can now proof our main result on the marked Wahl map for curves arising via the normalization of nodal curves on K3 surfaces. 
\begin{proof} [Proof of Theorem \ref{marked-wahl-k3}]
Assume $g-n \geq 13$ for $k=1$ or $g \geq 8$ for $k >1$, and let $ n \leq \frac{p(g,k)-2}{5}$. We wish to show that there is an irreducible component $I^0 \seq \mathcal{V}^n_{g,k}$ such that for a general $[(f: C \to X,L)] \in I^0$ the marked Wahl map $W_{C,T}$ is nonsurjective, where $T \seq C$ is the divisor over the nodes of $f(C)$.

Let $I \seq \mathcal{T}^n_{g,k}$ be the irreducible component from Theorem \ref{finiteness} (in the case $k=1$) or Theorem \ref{finiteness-nonprim} (in the case $k>1$) and set $I^0=I \cap \mathcal{V}^n_{g,k}$, which is nonempty from \cite[Thm.\ 2.8]{ded-sern}.
Let $\pi: Y \to X$ be the blow up of the K3 surface $X$ at the nodes of $f(C)$. There is a natural closed immersion $C \seq Y$. Let $E \seq Y$ denote the sum of the exceptional divisors, and let $M= \mathcal{O}_Y(C)$. We have $K_{C}=(M+E)_{|_{C}}$ by the adjunction formula.
Consider the following commutative diagram:
\[
\xymatrix{
\bigwedge^2 H^0(Y, M) \ar[r]^{\Phi_{M} \; \; \; \;}  \ar[d]^h & H^0(Y, \Omega_{Y}(M^2)) \ar[d]^g \\
\bigwedge^2 H^0(C, K_{C}(-T)) \ar[r]^{\; \; \; W_{C,T} \; \; \; \; }  &  H^0(C, K^3_{C}(-2T)).
}
\]
 where the top row is a Gaussian on $Y$. 
The map $h$ is surjective, as $H^1(Y, \mathcal{O}_Y)=0$.
Suppose for a contradiction that the marked Wahl map $W_{C,T}$ were surjective. Then $g$ would be surjective, and hence the natural map
$$ H^0(C, \Omega_{Y|_{C}}(M^2)) \to H^0(C, K^3_{C}(-2T))$$ would also be surjective.  Now consider the short exact sequence
$$0 \to M_{|_{C}} \to \Omega_{Y|_{C}}(M^2) \to K^3_{C}(-2T) \to 0 .$$ Since 
$$H^1(C,M)=H^1(C,K_{C}(-E))=H^0(C,E) \neq 0,$$
the surjectivity of $g$ would imply that $h^1(C,\Omega_{Y|_{C}}(M^2))=h^0(C,T_{Y|_{C}}(2E-K_{C})) \neq 0$.
However $$H^0(C,T_{Y|_{C}}(2E-K_{C})) \seq H^0(C, f^*(T_X)(2E-K_{C}))\seq H^0(C, f^*(T_X))=0$$
from Lemma \ref{coh-van}, and since $K_{C}-2E$ is effective for $n \leq \frac{p(g,k)-2}{5}$. So this is a contradiction and hence $W_{C,T}$ is nonsurjective.
\end{proof}

\begin{remark} \label{halic-remarks}
In the paper \cite{halic}, claims are made about the generic finiteness of $\eta \; : \; \mathcal{V}^n_{g,k}  \to \mathcal{M}_{p(g,k)-n}$ and the nonsurjectivity of the Wahl map for curves parametrized by the image of $\eta$. The proof of the first statement, \cite[Theorem 3.1]{halic}, seems flawed to us. Indeed, the statement in Step $1$ that $s=(s_0,0)$ is trivial, as $s$ defines the splitting.\footnote{This was pointed out to us by Stefan Schreieder.} The conclusion in Step 2 that family $(3.1)$ is trivial when restricted to a small analytic open subset of $\euf{S}$ seems likewise rather obvious, but in any case does not have the consequence claimed. The proof of Theorem $4.1$ also seems incorrect. Namely, the last row in diagram in $(A.1)$ should be twisted by $-2E$, but then Lemma $A.2$ fails.
\end{remark}

\section{Brill--Noether theory for nodal curves on K3 surfaces}
In this section we consider two related questions on the Brill--Noether theory of nodal curves on a K3 surface.  Let $D \seq X$ be a nodal curve on a K3 surface, and let $C:= \tilde{D}$ be the normalization of $D$. In the first part, we consider the Brill--Noether theory of the smooth curve $C$, whereas in the second part we consider the Brill--Noether theory for the nodal curve $D$.
\subsection{Brill--Noether theory for smooth curves with a nodal model on a K3 surface} \label{BNP-nodal}
In this section we will apply an argument from \cite{lazarsfeld-bnp} to the K3 surface $S_{p,h}$ as in Lemma \ref{onenodal-lem-a} in order to study the Brill--Noether theory for smooth curves with a primitive nodal model on a K3 surface.
\begin{lem}
Consider the K3 surface $S_{p,h}$ as in Lemma  \ref{onenodal-lem-a}. There is no expression $M=A_1+A_2$, where $A_1$ and $A_2$ are effective divisors with $h^0(Y,\mathcal{O}(A_1)) \geq 2$ and $h^0(Y,\mathcal{O}(A_2)) \geq 2$.
\end{lem}
\begin{proof}
We first claim that any effective divisor of the form $D=aR_1+bR_2$, for integers $a,b$, must have $a,b \geq 0$. Suppose for a contradiction that $a < 0$. Clearly we must have $b > 0$. Thus there is some integral component $D_1$ of $D$ with $(D_1 \cdot R_2) <0$, as $(D \cdot R_2)=-2b <0$. Thus $D_1 \sim R_2$. Repeating this argument on $D-R_2$, we see that $b R_2$ is a summand of $D$. But then $D-bR_2=aR_1$ is effective, which is a contradiction as $a<0$. Thus $a \geq 0$. Likewise $b \geq 0$. Furthermore, this argument also shows that all integral components of any effective divisor of the form $D=aR_1+bR_2$ are linearly equivalent to either $R_1$ or $R_2$. In particular, $D$ is rigid.

Suppose $M=A_1+A_2$ is an expression as above. Write $A_1=x_1M+\sum_{i=1}^2 y_{1,i} R_i$ and $A_2=x_2M+\sum_{i=1}^2 y_{2,i} R_i$ for integers $x_i,y_{i,j}$ for $i=1,2$, $1 \leq j \leq 2$. We have $x_1,x_2 \geq 0$ by Lemma  \ref{onenodal-lem-a} and $x_1+x_2=1$, and assume $x_1 \geq x_2$.  Thus we must have $x_2=0$, which gives  $h^0(Y,\mathcal{O}(A_2)) \leq 1$ (as the divisor $\sum_{i=1}^2 y_{2,i} \Gamma_i $ is rigid if $y_{2,i} \geq 0$ for all $i$ and not effective if there is some $j$ with $y_{2,j} <0$). 
\end{proof}

Let $C \seq X$ be a smooth curve on a K3 surface $X$. Let $M \in \text{Pic}(C)$ be a globally generated line bundle such that $\omega_C \otimes M^*$ is also globally generated. We denote by $F_{C,M}$ the vector bundle on $X$ defined as the kernel
of the evaluation map $H^0(C,M) \otimes_{\C} \mathcal{O}_X \twoheadrightarrow M$. Let $G_{C,M}$ be the dual bundle of $F_{C,M}$, this is globally generated from the exact sequence
$$ 0 \to H^0(M) \otimes \mathcal{O}_X \to G_{C,M} \to \omega_C \otimes M^* \to 0$$
(using $H^1(\mathcal{O}_X)=0$). The following generalization of \cite[Lemma 1.3]{lazarsfeld-bnp} is well-known, see \cite[Remark 3.1]{fkp-old}.
\begin{lem}
In the above situation, assume further that there is no expression $\mathcal{O}(C) \simeq L_1 \otimes L_2$, where $L_1$ and $L_2$ are effective line
bundles on $X$ with $h^0(L_i) \geq 1+s_i$ for $i=1,2$, where $s_i \geq 1$ are integers satisfying $s_1+s_2=h^0(C,M)$. Then $F_{C,M}$ is a simple vector bundle.
\end{lem}
\begin{proof}
We follow the proof of \cite[Ch.7, Prop.2.2]{huy-lec-k3}. The bundle $F_{C,M}$ is simple if and only if its dual $G_{C,M}$ is simple. Suppose $G_{C,M}$ were not simple. Then there would exist a non-trivial endomorphism $\psi: G \to G$ with nontrivial kernel. Set $K:= \text{im}(\psi)$, $L_1:=\det(K)$ and $L_2:=\det((G/K))$. Set $s_1:= \rank(K)$ and $s_2:=\rank((G/K)/T)$, where $T$ is the maximal torsion subsheaf of $G/K$. Clearly $s_i \geq 1$ for $i=1,2$ and $s_1+s_2=\rank(G)=h^0(C,M)$. So it suffices to prove $h^0(L_i) \geq 1+s_i$ for $i=1,2$. As is explained in \cite[Sec.7, Prop.2.2]{huy-lec-k3}, if we pick a sufficiently positive divisor $D$ on $X$ we have $h^0(D,L_i|_D) \geq s_i +1$ (as $c_1(T)$ is effective). On the other hand, if $D$ is sufficently positive  then $D-L_i$ is big and nef, so that $H^0(X,L_i) \twoheadrightarrow H^0(D,L_i|_D)$ and thus $h^0(L_i) \geq 1+s_i$ for $i=1,2$.
\end{proof}

\begin{cor} \label{BNP-cor}
Consider a K3 surface $S_{p,h}$ as in Lemma  \ref{onenodal-lem-a}. Let $C \in |M|$ be a general smooth curve. Then $C$ is Brill--Noether--Petri general.
\end{cor}
\begin{proof}
This follows immediately from the proof of the main theorem in \cite{lazarsfeld-bnp} and the above lemma.
\end{proof}
Putting all the pieces together, we can prove Proposition \ref{BNP-theorem}.
\begin{proof} [Proof of Proposition \ref{BNP-theorem}]
Assume $g-n \geq 8$. We must show that there exists a component $\mathcal{J} \seq \mathcal{V}^n_g$ such that for $[(f:C \to X,L)] \in \mathcal{J}$ general, $C$ is Brill--Noether--Petri general.

Set $h=g-n$, $p=g$. The case $n=0$ is \cite{lazarsfeld-bnp}, so we may assume $p>h$.  Let $l,m$ be the unique nonnegative integers such that $$p-h= \left \lfloor \frac{h+1}{2} \right \rfloor l+m $$
and $0 \leq m < \left \lfloor \frac{h+1}{2} \right \rfloor $. Set $\epsilon=1$ if $m=0$ or $m=\left \lfloor \frac{h+1}{2} \right \rfloor -1$ and $\epsilon=0$ otherwise. Then $(M+R_1+\epsilon R_2)^2=2g-2$, where $M$, $R_1$, $R_2$ are a basis of $P_{p,h}$ as in Lemma  \ref{onenodal-lem-a}. The claim then follows from the proof of Theorem \ref{finiteness}, by deforming to the curve $$R:=D \cup R_1 \cup \epsilon R_2 $$ on $S_{p,h}$, where $D \in |M|$ is general, marked at all nodes other than one point from $D \cap R_i$ for $i=1,2$. Note that the partial normalization of $R$ at the marked nodes is an unstable curve, and the stabilization is isomorphic to $D$, which is Brill--Noether--Petri general by Corollary \ref{BNP-cor}.
\end{proof}

\subsection{Brill--Noether theory for nodal rational curves on K3 surfaces} \label{rat}
In this section we will denote by $X$ a K3 surface with 
$\text{Pic}(X) \simeq \mathbb{Z}L$, $(L)^2=2g-2$ with $g \geq 2$, and $C \in |L|$ will denote a fixed \emph{rational} curve (not necessarily nodal). Let $\bar{J}^d(C)$ denote the compacified Jacobian of degree $d$, rank one torsion free sheaves and consider the \emph{generalized Brill--Noether loci}
$$\overline{W}^r_d(C) := \{ \text{$A \in \bar{J}^d(C)$ with $h^0(A) \geq r+1$} \}$$ which can be given a determinantal scheme structure, see \cite{bhosle-parawes}. There is an open subset $W^r_d(C) \seq \overline{W}^r_d(C) $ parametrizing line bundles. We will denote by $\rho(g,r,d)$ the Brill--Noether number $g-(r+1)(g-d+r)$.

The following comes from the proof of \cite[Remark 2.3(i)]{bhosle-parawes} (although it may have been known to experts earlier). The proof is essentially the same as in the smooth case.
\begin{thm} \label{bhosle-bn}
Each irreducible component of $\overline{W}^r_d(C)$ has dimension at least $\rho(g,r,d)$.
\end{thm}
In the case $\rho(g,r,d) \geq 0$, $\overline{W}^r_d(C)$ is nonempty, \cite[Thm.\ 3.1]{bhosle-parawes}. If $\rho(g,r,d) > 0$, then under our hypotheses $\overline{W}^r_d(C)$ is connected, \cite[Thm.\ 1]{gomez}.

Let $V_{d,r} \seq \overline{W}_d^r(C) $ be the open locus parametrizing sheaves $A$ which are globally generated and with $h^0(A)=r+1$. Assume $V_{d,r} \neq \emptyset$. We will begin by proving that $\dim V_{d,r} \leq \rho(g,r,d)$ (in particular $V_{d,r}= \emptyset$ if $\rho(g,r,d) < 0$). 

Fix a vector space $\mathbb{H}$ of dimension $r+1$ and let 
$P^r_d \to V_{d,r}$ parametrize pairs $(A,\lambda)$ where $A \in V_{d,r}$ and $\lambda$ is a surjection of $\mathcal{O}_X$ modules
$$ \lambda : \mathbb{H} \otimes \mathcal{O}_X \to A$$ inducing an isomorphism $\mathbb{H} \simeq H^0(A)$. Two such surjections are identified if they differ by multiplication by a nonzero scalar. Thus $P^r_d$ is a PGL(r+1) bundle over $V_{d,r}$.

Let $(A,\lambda) \in P^r_d$. Then $Ker \lambda$ is a vector bundle $F$ of rank $r+1$, \cite[\S3.2]{gomez}. We have $det(F) \simeq L^*$, $deg(c_2(F))=d$, $h^0(F)=h^1(F)=0$ and $h^2(F)=r+1+(g-d+r)$, cf. \cite[\S1]{lazarsfeld-bnp}. Note that $g-d+r =h^1(A) \geq 0$. Further, for any rank one, torsion-free sheaf $A$ on $C$ we may define an `adjoint' $A^{\dagger}$, which is a rank one torsion-free sheaf with $(A^{\dagger})^{\dagger}=A$. From the short exact sequence
$$0 \to F \to  \mathbb{H} \otimes \mathcal{O}_X \to A \to 0$$
we may form the dual sequence
$$0 \to \mathbb{H}^* \otimes \mathcal{O}_X \to F^* \to A^{\dagger} \to 0.$$

The following lemma is a slight generalization of \cite[Cor.\ 9.3.2]{huy-lec-k3}:
\begin{lem}
Assume $\text{Pic}(X) \simeq \mathbb{Z}L$ as above and let $(A,\lambda) \in P^r_d$. Then the vector bundle $F=Ker \lambda$ is stable.
\end{lem}
\begin{proof}
For any vector bundle $H \seq \mathcal{O}_X^{\oplus a}$ and any integer $s \geq 1$, we have $h^0(\bigwedge^s H^*) \geq 1$. Indeed, we have $\bigwedge^s H \seq \bigwedge^s \mathcal{O}_X^{\oplus a} \simeq  \mathcal{O}_X^{\oplus b}$ for some integer $b$ and then $\mathcal{E}nd_{\mathcal{O}_X}(\bigwedge^s H) \simeq \bigwedge^s H \otimes \bigwedge^s H^* \seq (\bigwedge^s H^*)^{\oplus b}$. Taking global sections gives $h^0(\bigwedge^s H^*) \geq 1$ (as $id \in H^0(\mathcal{E}nd_{\mathcal{O}_X}(\bigwedge^s H)))$.

Now let $F' \seq F \seq \mathbb{H} \otimes \mathcal{O}_X $ be a locally free subsheaf  of $F$ with $rk(F')=r'<r+1$. From the above, $h^0(det(F'^*)) \geq 1$ and $h^0(\bigwedge^{r'-1}F'^*) \geq 1$ (if $r' >1$). As $\text{Pic}(X) \simeq \mathbb{Z}L$, we have $det(F')=kL^*$ for some $k \geq 0$. We claim $k \neq 0$. If $k=0$ and $r'=1$, then $F' \simeq \mathcal{O}_X$ which contradicts that $h^0(F)=0$. If $k=0$, $r' >0$, then $F' \simeq \bigwedge^{r'-1}F'^* \otimes det(F')$ gives $h^0(F') \geq 1$ which is again a contradiction. So we have $det(F')=kL^*$ for $k >0$ and $det(F)=L^*$, which implies $\deg(F')/ rk(F') < \deg(F) / rk(F)$ as required.
\end{proof}

Let $M_v$ be the moduli space of stable sheaves on $X$ with Mukai vector $v=(r+1,L,g-d+r)$. We have a morphism 
\begin{align*}
\psi_C: P^r_d &\to M_v \\
(A, \lambda) &\mapsto (Ker \lambda)^*
\end{align*} where $(Ker \lambda)^*$ denotes the dual bundle to $Ker \lambda$.
Let $M_C$ be the closure of the image of $\psi_C$, with the induced reduced scheme structure. By the description of $F$, if $[F^*] \in Im(\psi_C)$, $$c_2(F^*) \sim d c_X$$ where $c_X$ is 
the rational equivalence class of a point lying on a rational curve as defined in \cite{bv-chow}.

There is a natural symplectic form $\alpha$ on $M_v$ defined in \cite{mukai-symplectic}.
\begin{prop}
Let $\alpha$ be the natural symplectic form on $M_v$ and $i: M^0_C \to M_v$ the inclusion, where $M^0_C$ is the smooth locus of $M_C$. Then $i^*\alpha=0$.
\end{prop}
\begin{proof}
Since $g-d+r \geq 0$, \cite[Thm.\ 0.6(1)]{ogrady} applies and for any $[G] \in M_v$, there is an effective, degree $\rho(g,r,d)$ zero-cycle $Z$ with $c_2(G) \sim [Z]+ac_X$ for some $a \in \mathbb{Z}$.  Following \cite[Prop.\ 1.3]{ogrady} there is then a smooth quasi-projective variety $\widetilde{M}_v$ with morphisms $q: \widetilde{M}_v \to M_v$, $p:\widetilde{M}_v \to X^{[\rho(g,r,d)]}$  such that $q$ is surjective and generically finite, and with the property that if $x=[F^*] \in Im(\psi_C)$ and $y \in q^{-1}(x)$, then we have the rational equivalence 
\begin{align} \label{rateq}
p(y)+(d-\rho(g,r,d))c_X \sim c_2(F^*) \sim d c_X.
\end{align} Further if $\beta$ is the symplectic form on the Hilbert scheme of points $X^{[\rho(g,r,d)]}$, we have $q^* \alpha=k p^* \beta$ for some nonzero constant $k \in \C$. 
Let $\widetilde{M}_C \seq \widetilde{M}_v$ denote the closure of $q^{-1}(Im(\psi_C))$ with the induced reduced scheme structure, and let $p_C$ respectively $q_C$ be the restriction of $p$ respectively $q$ to the smooth locus of $\widetilde{M}_C$.  Then $p_C(s)$ is rationally equivalent to $p_C(t)$ for all $s, t$ in the smooth locus of $\widetilde{M}_C$, from (\ref{rateq}). Thus $p_C^*(\beta)=0$ by\cite{mumf-chow}. Hence $q_C^*(\alpha)=0$ and since $q$ is surjective, $i^*\alpha=0$.
\end{proof}

\begin{cor}
We have $\dim M_C \leq \rho(g,r,d)$.
\end{cor}
\begin{proof}
Indeed $\dim M_v=2\rho(g,r,d)$ from \cite[Thm.\ 0.1]{mukai-symplectic} so this follows from the proposition above.
\end{proof}

\begin{cor} \label{tf-bn}
If  $V_{d,r}$ is nonempty then $\dim V_{d,r} = \rho(g,r,d)$.
\end{cor}
\begin{proof}
If $V_{d,r}$ is nonempty then $\dim V_{d,r} \geq \rho(g,r,d)$ by Theorem \ref{bhosle-bn}, so it suffices to show $\dim V_{d,r} \leq \rho(g,r,d)$. It then suffices to show that $\psi_C: P^r_d \to M_v$ has fibres of dimension $\dim\text{PGL(r+1)}$. In other words, we need to show that for each fixed $F \in M_v$ there are only finitely many
$A \in V_{d,r}$ fitting into an exact sequence $0 \to F \to \mathbb{H} \otimes \mathcal{O}_X \to A \to 0$. But this follows immediately from the fact that in our circumstances the degeneracy locus map $Gr(r+1,H^0(F^*)) \to |L|$ is globally defined and finite, see \cite[\S2]{ogrady} (recall that all such $A$ are supported on a fixed $C$ by definition).
\end{proof}
\begin{rem}
Assume $V_{d,r}$ is nonempty. We have $\dim M_C = \dim V_{d,r}=\rho(g,r,d)$ from the above Corollary. Thus $M_C$ is a (possibly singular) Lagrangian subvariety of $M_v$. 
\end{rem}

\begin{cor} \label{bnempty}
Let $X$ be a K3 surface with $Pic(X) \simeq \mathbb{Z} L$ and $(L \cdot L)=2g-2$. Let $C \in |L|$ be rational and assume $\rho(g,r,d) <0$. Then
$$\overline{W}^r_d(C) := \{ \text{$A \in \bar{J}^d(C)$ with $h^0(A) \geq r+1$} \}$$ is empty.
\end{cor} 
\begin{proof}
Assume for a contradiction that $A \in \overline{W}^r_d(C)$. Let $A'$ be the image of the evaluation morphism $H^0(A) \otimes \mathcal{O}_C \to A$. Then $A'$ is a globally generated, torsion free, rank one sheaf of degree $d' \leq d$ with $r'+1 \geq r+1$ sections, and thus $A' \in V_{d',r'}$. But $\rho(g,r',d') \leq \rho(g,r,d) <0$ for $d' \leq d$, $r' \geq r$ and thus $V_{d',r'}$ is empty by Corollary \ref{tf-bn}. This is a contradiction. 
\end{proof}

To proceed we need two technical lemmas.
\begin{lem} \label{techlem}
Let $C$ be an arbitrary integral \textbf{nodal} curve. Suppose $A'$ is a rank one torsion free sheaf on $C$ and let $k(p)$ be the length one skyscraper sheaf on $C$ supported at a node $p \in C$. Then if $\mathcal{Z} \seq \overline{W}_{d}^r(C)$ is an irreducible family of rank one torsion free sheaves such that we have an exact sequence
$$0 \to A' \to A \to k(p) \to 0 $$
for all $A \in \mathcal{Z}$, then $\dim \mathcal{Z} \leq 1$.
\end{lem}
\begin{proof}
It suffices to show $\dim_{\C} Ext^1_{\mathcal{O}_C}(k(p), A') \leq 2$.  We have
\begin{align*}
Ext^1_{\mathcal{O}_C}(k(p), A') &\simeq Ext^1_{\mathcal{O}_C}(k(p), A'(n)) \text{\; for any $n\in \mathbb{Z}$} \\
&\simeq H^0(C,\mathcal{E}xt^1_{\mathcal{O}_C}(k(p), A'(n))) \text{\; for $n\gg 0$} \\
&\simeq H^0(C,\mathcal{E}xt^1_{\mathcal{O}_C}(k(p), A'))
\end{align*}
where the second line follows from \cite[Prop III.6.9]{har}. The sheaf $\mathcal{E}xt^1_{\mathcal{O}_C}(k(p), A')$ is a skyscraper sheaf supported at $p$. If $A'$ is nonsingular at $p$ then
$$\dim_{\C} Ext^1_{\mathcal{O}_C}(k(p), A') = \dim_{\C} Ext^1_{\mathcal{O}_C}(k(p), \omega_C)=1,$$ by Serre duality. 

Suppose now $A'$ is singular at $p$ and let $\pi: C' \to C$ be the normalization of $C$. Then
$\dim_{\C} Ext^1_{\mathcal{O}_C}(k(p), A') = \dim_{\C} Ext^1_{\mathcal{O}_C}(k(p), \pi_*(\mathcal{O}_{C'}))$ since $A'_p \simeq m_p \simeq \pi_*(\mathcal{O}_{C'})_p$ where $m_p$ is the maximal ideal of $p$, by \cite[III.1]{cpt-jac-2}\footnote{Note that $m_p$ is a degree $-1$, rank one t.f.\ sheaf, and $\pi_*(\mathcal{O}_{C'})$ is a t.f.\ sheaf of strictly positive degree, so although these sheaves are locally isomorphic, they are \emph{not} globally isomorphic.}. But
$\dim_{\C} Ext^1_{\mathcal{O}_C}(k(p),\pi_*(\mathcal{O}_{C'}))=2$ as required, by \cite[Prop.\ 2.3]{av-lang}. 
\end{proof}

\begin{lem} \label{techlem2}
Let $C$ be an arbitrary integral nodal curve. Suppose $A'$ is a rank one torsion free sheaf and let $Q$ be a sheaf with zero-dimensional support such that $supp(Q) \seq C_{\text{sing}}$, where $C_{\text{sing}}$ is the singular locus of $C$. Then if $\mathcal{Z} \seq \overline{W}_{d}^r(C)$ is an irreducible family of rank one torsion free sheaves such that we have an exact sequence
$$0 \to A' \to A \to Q \to 0 $$
for all $A \in \mathcal{Z}$, then $\dim \mathcal{Z} \leq l(Q)$, where $l(Q)$ denotes the length of $Q$.
\end{lem}
\begin{proof}
We will prove the result by induction on $l(Q)$. When $l(Q)=1$ the result holds from
Lemma \ref{techlem}. Suppose $Q$ has length $r$ and choose a sheaf $Q'$ with zero-dimensional support and length $r-1$ such that we have a surjection $\phi : Q \twoheadrightarrow Q'$. For any $$0 \to A' \to A \to Q \to 0, $$ $\phi$ then induces
a short exact sequence
$$0 \to A'' \to A \to Q' \to 0$$
where $A''$ fits into the exact sequence
$$0 \to A' \to A'' \to Ker(\phi) \to 0.$$ Now let $\pi: \mathcal{T} \to \mathcal{Z}$ be the moduli space with fibre over $A \in \mathcal{Z}$ parametrising all extensions
$0 \to A' \to A \to Q \to 0$; this can be constructed from \cite[$\S  4.1$]{bridgeland}. After replacing $\mathcal{T}$ with an open set we have a morphism $\psi : \mathcal{T} \to \overline{W}_{d'}^{r'}(C)$ for some $d'$, $r'$, which sends a point representing the exact sequence $0 \to A' \to A \to Q \to 0$ to $A'':=Ker(A \to Q')$. By Lemma \ref{techlem} the image of $\psi$ is at most one dimensional as $l(Ker(\phi))=1$. Further, $\pi(\psi^{-1}(A''))$ is at most $l(Q')=l(Q)-1$ dimensional for any $A'' \in Im(\psi)$ by the induction hypothesis. It then follows than $\dim \mathcal{Z} \leq l(Q)$ as required.
\end{proof}

\begin{lem} \label{zero-case}
Let $C$ be an integral, nodal curve. Then $$\overline{W}^0_d(C) := \{ \text{$A \in \bar{J}^d(C)$ with $h^0(A) \geq 1$} \}$$ is irreducible of dimension $\rho(g,0,d)=d$.
\end{lem}
\begin{proof}
Let $U_d \seq \overline{W}^0_d(C)$ be the open subset consisting of line bundles. Let $V_d:=Div_d(C)$ denote the 
scheme parametrizing zero-dimensional schemes $Z \seq C$ such that the ideal sheaf $I_Z$ is invertible of degree $d$; i.e\ $V_d$ is the scheme of effective Cartier divisors. Let $\tilde{C} \to C$ denote the normalization. From \cite[Thm.\ 2.4]{kleiman-special}, pullback induces a birational morphism $V_d \to Div_d(\tilde{C})$, and thus $\dim V_d=d$. We have a morphism $V_{d} \to \bar{J}^d(C)$ with image $U_d$, which sends a scheme $Z$ to the effective line bundle $I_Z^*$. Thus $U_d$ of dimension at most $d$. Since each component of $\overline{W}^0_d(C)$ has dimension at least $d$ by Theorem \ref{bhosle-bn}, we see $\dim(U_d) =d$.

Let $I$ be an irreducible component of $\overline{W}^0_d(C) \setminus U_d$; we need to prove $\dim(I) < d$. There is a nonempty open set $I^0$ of $I$, an integer $d' < d$ and a partial normalization $\mu: C' \to C$ such that for each $A \in I^0$ there exists a unique effective line bundle $B \in Pic^{d'}(C')$ with $\mu_*(B) \simeq A$ by \cite[Prop.\ 3.4]{gagne}. Since the dimension of the moduli space of effective line bundles of degree $d'$ on $C'$ has dimension $d'$ by the above, we see that $\dim(I) \leq d' < d$.
\end{proof}

We now prove the main result of this section, which states that if $X$ is a K3 surface with $Pic(X) \simeq \mathbb{Z} L$ and $(L \cdot L)=2g-2$, and if $D \in |L|$ is a rational nodal curve, then
$$\overline{W}^r_d(D) := \{ \text{$A \in \bar{J}^d(D)$ with $h^0(A) \geq r+1$} \}$$ is either empty or is equidimensional of the expected dimension $\rho(g,r,d)$.
\begin{proof} [Proof of Theorem \ref{bn-rat-thm}]
By Corollary \ref{bnempty} the theorem holds whenever $\rho(g,r,d) <0$. Thus it suffices to prove the theorem for $\rho(g,r,d)  \geq -1$. We will proceed by induction on $\rho(g,r,d)$ starting from the case $\rho(g,r,d)  = -1$.

Choose nonnegative integers $r,d$ and suppose the claim holds for all $r',d'$ such that $\rho(g,r',d') < \rho(g,r,d)$. We know the claim holds for $r=0$ by Lemma \ref{zero-case}, so we may suppose $r >0$.   Let $I$ be an irreducible component of $\overline{W}^r_d(C) \setminus V_{d,r}$; from Theorem \ref{bhosle-bn} it suffices to show $\dim(I) < \rho(g,r,d)$. For all $A \in I$, we denote by $A'$ the globally generated part of $A$, i.e.\ the image of the evaluation morphism $H^0(A) \otimes \mathcal{O}_C \to A$. There is an open dense subset $I^0 \seq I$ such that $deg(A')=d'$, $h^0(A')=r'$  is constant for all $A \in I^0$. Replacing $I^0$ by a smaller open set if necessary, we have a morphism 
\begin{align*}
f \;: \; I^0_{\text{red}} &\to \overline{W}^{r'}_{d'}(C) \\
A & \mapsto A'.
\end{align*}
Indeed, let $S$ be an integral, locally Noetherian scheme over $\C$, let $\pi: C \times S  \to S$ be the projection, and let $\mathcal{A} $ be an $S$ flat family of rank one torsion free sheaves $\mathcal{A}_s$ on $C$, with $\deg(\mathcal{A}_s)=d$, $h^0(\mathcal{A}_s)=r'+1$ constant. Replacing $S$ with an open subset, we may assume that $\pi_* \mathcal{A}$ is a trivial vector bundle of rank $r'+1$, and that the image $\mathcal{A}'$ of the evaluation morphism $$H^0(\mathcal{A}) \otimes \mathcal{O}_{C \times S} \to \mathcal{A} $$ is flat over $S$. Replacing $S$ with another open subset, we may further assume $\pi_* \mathcal{A}'$ is a trivial vector bundle of rank $r'+1$. We claim that $\mathcal{A}'_s$ is the base-point free part of $\mathcal{A}_s$.  Let $B_s$ denote the base point free part of $\mathcal{A}_s$. The surjection
$H^0(\mathcal{A}_s) \otimes \mathcal{O}_{C} \to \mathcal{A}'_s$ shows $\mathcal{A}'_s \seq B_s \seq A$. Then the exact sequence
$$ 0 \to \mathcal{A}'_s \to B_s \to F \to 0,$$
where $F$ has zero-dimensional support, and the equality $h^0(\mathcal{A}'_s)=h^0(B_s)=r+1$ implies $F$ is the zero sheaf (since $B_s$ is base point free). Thus if $d':= \deg(\mathcal{A}'_s)$, $\mathcal{A}'$ is a flat family of rank one, torsion free sheaves on $C$ of degree $d'$ with $r'+1$ sections, so the universal property of  $\overline{W}^{r'}_{d'}(C)$ induces a morphism $S \to \overline{W}^{r'}_{d'}(C)$.

We  next claim that $f$ has fibres of dimension at most $d-d'$. This will then imply the result as $\rho(g,r',d') \leq \rho(g,r,d')= \rho(g,r,d)- (r+1)(d-d')$ so that $\dim(I^0) < \rho(g,r,d)$ for $r \neq 0$. For any $A \in I^0$ we have a sequence
$$0 \to f(A) \to A \to Q_A \to 0$$
where $Q_A$ has zero-dimensional support. We have a canonical decomposition
$Q_A = Q_{A,\text{sm}} \oplus Q_{A,\text{sing}}$ with $\text{Supp}(Q_{A,\text{sm}}) \seq C_{\text{sm}}$ and $\text{Supp}(Q_{A,\text{sing}}) \seq C_{\text{sing}}$,
where $C_{\text{sm}}$ is the smooth locus of $C$ and $C_{\text{sing}}=C-C_{\text{sm}}$. Replacing $I^0$ with a dense open set we may assume $Q':=Q_{A,\text{sing}}$ is independent of $A \in I^0$. Let $e:=l(Q')$. For any $A \in I^0$, there is a unique effective line bundle $M$ of degree $d-d'-e$ such that we have a short exact sequence
$$0 \to f(A)(M) \to A \to Q' \to 0 .$$
We have a morphism
\begin{align*}
g \;: \; I^0 &\to \overline{W}^{r'}_{d-e}(C) \\
A & \mapsto f(A)(M).
\end{align*}
By Lemma \ref{techlem2}, $g$ has fibres of dimension at most $e$. For any $A'$ in the image of $f$ consider
$$ g|_{f^{-1}(A')} \;: \; f^{-1}(A') \to \overline{W}^{r'}_{d-e}(C) .$$ The image of
$ g|_{f^{-1}(A')} $ is a subset of the space of tuples $A' \otimes M$ for $M \in Pic^{d-d'-e}(C)$ effective. The moduli space of effective line bundles in $Pic^{d-d'-e}(C)$ may be identified with the image of the natural map $C^{[d-d'-e]}_{\text{sm}} \to Pic^{d-d'-e}(C)$, where $C_{\text{sm}}$ is the smooth locus of $C$, and thus has dimension at most $d-d'-e$. Thus $\dim f^{-1}(A')\leq d-d'$ as required.

\end{proof}
\begin{rem}
It is clear from the proof that the theorem would hold for any constant cycle curve $C \in |L|$ such that $C$ is integral and nodal, see \cite{huy-const-cyc}.
\end{rem}

\bibliography{biblio}

\begin{thebibliography}{10}

\bibitem{arakol}
C.~Araujo and J.~Koll\'{a}r.
\newblock Rational curves on varieties.
\newblock In {\em Higher Dimensional Varieties and Rational Points}. Springer,
  Berlin, 2003.

\bibitem{sernesi-brill}
E.~Arbarello, A.~Bruno, and E.~Sernesi.
\newblock {Mukai's program for curves on a K3 surface}.
\newblock arXiv:1309.0496 (to appear in Algebraic Geometry).

\bibitem{arbarello-II}
E.~Arbarello, J.~D. Harris, M.~Cornalba, and P.~Griffiths.
\newblock {\em Geometry of Algebraic Curves: Volume II with a contribution by
  Joseph Daniel Harris}, volume 268.
\newblock Springer, 2011.

\bibitem{av-lang}
D.~Avritzer, H.~Lange, and F.~A. Ribeiro.
\newblock Torsion-free sheaves on nodal curves and triples.
\newblock {\em Bull. Braz. Math. Soc.}, 41(3):421--447, 2010.

\bibitem{barth}
W.~Barth, C.~Peters, and A.~Van de~Ven.
\newblock {\em Compact Complex Surfaces}.
\newblock Springer, Berlin, 2004.

\bibitem{bv-chow}
A.~Beauville and C.~Voisin.
\newblock {On the Chow ring of a K3 surface}.
\newblock {\em J. Alg. Geom.}, 13(3):417--426, 2004.

\bibitem{bhosle-parawes}
U.~N. Bhosle and A.~J. Parameswaran.
\newblock {Picard bundles and Brill--Noether loci in the compactified Jacobian
  of a nodal curve}.
\newblock {\em Int. Math. Res. Not.}, 2013.

\bibitem{bogomolov}
F.~Bogomolov, B.~Hassett, and Y.~Tschinkel.
\newblock {Constructing rational curves on K3 surfaces}.
\newblock {\em Duke Math. J.}, 157(3):535--550, 2011.

\bibitem{bridgeland}
T.~Bridgeland.
\newblock {An introduction to motivic Hall algebras}.
\newblock {\em Adv. Math.}, 229(1):102--138, 2012.

\bibitem{buchweitz-flenner}
R-O Buchweitz and H.~Flenner.
\newblock A semiregularity map for modules and applications to deformations.
\newblock {\em Comp. Math.}, 137(2):135--210, 2003.

\bibitem{chen-rational}
Xi~Chen.
\newblock {Rational curves on $K3$ surfaces}.
\newblock {\em J. Algebr. Geom.}, 8(2):245--278, 1999.

\bibitem{chiang-lipman}
H.J. Chiang-Hsieh and J.~Lipman.
\newblock A numerical criterion for simultaneous normalization.
\newblock {\em Duke Math. J.}, 133(2):347--390, 2006.

\bibitem{ciliberto-knutsen-gonal}
C.~Ciliberto and A.L. Knutsen.
\newblock {On {$k$}-gonal loci in Severi varieties on general K3 surfaces and
  rational curves on hyperk\"{a}hler manifolds}.
\newblock {\em J. Math. Pures. Appl.}, 2013.

\bibitem{ciliberto-lopez-miranda}
C.~Ciliberto, A.~Lopez, and R.~Miranda.
\newblock {Projective degenerations of K3 surfaces, Gaussian maps, and Fano
  threefolds}.
\newblock {\em Invent. Math.}, 114(1):641--667, 1993.

\bibitem{cili-corank}
C.~Ciliberto, A.~F. Lopez, and R.~Miranda.
\newblock {On the corank of Gaussian maps for general embedded K3 surfaces}.
\newblock In {\em {Proceedings of the Hirzebruch 65 Conference on Algebraic
  Geometry (Ramat Gan, 1993)}}, volume~9 of {\em Israel Math. Conf. Proc.},
  pages 141--157, Ramat Gan, 1996. Bar-Ilan Univ.

\bibitem{cili-classification}
C.~Ciliberto, A.~F. Lopez, and R.~Miranda.
\newblock Classification of varieties with canonical curve section via
  {G}aussian maps on canonical curves.
\newblock {\em Amer. J. Math.}, 120(1):1--21, 1998.

\bibitem{wahl-plane-nodal}
C.~Ciliberto, A.~F. Lopez, and R.~Miranda.
\newblock On the {W}ahl map of plane nodal curves.
\newblock In {\em Complex analysis and algebraic geometry}, pages 155--163. de
  Gruyter, Berlin, 2000.

\bibitem{cossec-dolgachev}
F.~Cossec and I.~Dolgachev.
\newblock {\em Enriques surfaces I}.
\newblock Springer, 1989.

\bibitem{hirsch-zeit}
J.~d'Almeida and A.~Hirschowitz.
\newblock Quelques plongements projectifs non sp\'eciaux de surfaces
  rationnelles.
\newblock {\em Math. Z.}, 211(3):479--483, 1992.

\bibitem{dedieu}
T.~Dedieu.
\newblock {Severi varieties and self-rational maps of K3 surfaces}.
\newblock {\em Internat. J. Math.}, 20(12):1455--1477, 2009.

\bibitem{ded-sern}
T.~Dedieu and E.~Sernesi.
\newblock Equigeneric and equisingular families of curves on surfaces, 2014.
\newblock arXiv:1410.4221.

\bibitem{dolgachev}
I.V. Dolgachev.
\newblock Mirror symmetry for lattice polarized {$K3$} surfaces.
\newblock {\em J. Math. Sci.}, 81(3):2599--2630, 1996.

\bibitem{donagi-morrison}
R.~Donagi and D.~Morrison.
\newblock {Linear systems on K3 sections}.
\newblock {\em J. Diff. Geom.}, 29(1):49--64, 1989.

\bibitem{cpt-jac-2}
C.~D'Souza.
\newblock {Compactification of generalised Jacobians}.
\newblock {\em Proc. Math. Sci.}, 88(5):421--457, 1979.

\bibitem{fkp-old}
F.~Flamini, A.L. Knutsen, and G.~Pacienza.
\newblock {Singular curves on a K3 surface and linear series on their
  normalizations}.
\newblock {\em Int. J. Math}, 18(6):671--693, 2007.

\bibitem{flam}
F.~Flamini, A.L. Knutsen, G.~Pacienza, and E.~Sernesi.
\newblock {Nodal curves with general moduli on K3 surfaces}.
\newblock {\em Comm. Algebra}, 36(11):3955--3971, 2008.

\bibitem{flenner-ueber}
H.~Flenner.
\newblock {\em {{\"U}ber Deformationen holomorpher Abbildungen}}.
\newblock Habilitationsschrift, Universit{\"a}t Osnabr{\"u}ck, 1979.
\newblock Copy available at
  \url{http://www.ruhr-uni-bochum.de/imperia/md/content/mathematik/lehrstuhli/deformationen.pdf}.

\bibitem{fried}
R.~Friedman and J.~Morgan.
\newblock {\em {Smooth four-manifolds and complex surfaces}}.
\newblock Springer, Berlin, 1994.

\bibitem{fulpar}
W.~Fulton and R.~Pandharipande.
\newblock Notes on stable maps and quantum cohomology, 1996.
\newblock arXiv:alg-geom/9608011v2.

\bibitem{gagne}
M.~Gagne.
\newblock {\em {Compactified Jacobians of integral curves with double points}}.
\newblock PhD thesis, Massachusetts Institute of Technology, 1997.

\bibitem{galati-knutsen}
C.~Galati and A.L. Knutsen.
\newblock {On the existence of curves with $A_k$-singularities on K3 surfaces},
  2014.
\newblock arXiv:1107.4568, (to appear in Math. Res. Lett.).

\bibitem{gomez}
T.~L. Gomez.
\newblock {\em {Irreducibility of the moduli space of vector bundles on
  surfaces and Brill--Noether theory on singular curves}}.
\newblock PhD thesis, Princeton, 2000.

\bibitem{ghs-rat}
T.~Graber, J.~Harris, and J.~Starr.
\newblock Families of rationally connected varieties.
\newblock {\em J. AMS}, 16:57--67, 2003.

\bibitem{greuel}
G.M. Greuel, C.~Lossen, and E.~Shustin.
\newblock Geometry of families of nodal curves on the blown-up projective
  plane.
\newblock {\em Trans. AMS}, 350(1):251--274, 1998.

\bibitem{grothendieck-brauer}
A.~Grothendieck.
\newblock {Le groupe de Brauer: I. Alg{\`e}bres d'Azumaya et
  interpr{\'e}tations diverses}.
\newblock {\em S{\'e}minaire Bourbaki}, 9:199--219, 1964.

\bibitem{halic}
M.~Halic.
\newblock {Modular properties of nodal curves on K3 surfaces}.
\newblock {\em Math. Z.}, 270(2-3):871--887, 2012.

\bibitem{harrissev}
J.~Harris.
\newblock {On the Severi problem}.
\newblock {\em Invent. Math.}, 84(3):445--461, 1986.

\bibitem{harris-morrison}
J.~Harris and I.~Morrison.
\newblock {\em Moduli of curves}, volume 187.
\newblock Springer New York, 1998.

\bibitem{har}
R.~Hartshorne.
\newblock {\em Algebraic Geometry}.
\newblock Springer, New York, 1977.

\bibitem{hirschowitz}
A.~Hirschowitz.
\newblock Une conjecture pour la cohomologie des diviseurs sur les surfaces
  rationnelles g\'en\'eriques.
\newblock {\em J. Reine Angew. Math.}, 397:208--213, 1989.

\bibitem{huy-lec-k3}
D.~Huybrechts.
\newblock {Lectures on K3 surfaces}.
\newblock Available at
  \url{http://www.math.uni-bonn.de/people/huybrech/K3Global.pdf}.

\bibitem{huy-const-cyc}
D.~Huybrechts.
\newblock {Curves and cycles on K3 surfaces}.
\newblock {\em Algebraic Geometry}, 1(1):69--106, 2014.

\bibitem{huy-kem}
D.~Huybrechts and M.~Kemeny.
\newblock {Stable Maps and Chow Groups}.
\newblock {\em Doc.\ Math.}, 18:507--517, 2013.

\bibitem{huybrechts-sheaves}
D.~Huybrechts and M.~Lehn.
\newblock {\em The geometry of moduli spaces of sheaves}.
\newblock Cambridge, 2010.

\bibitem{huybrechts-brauer}
D.~Huybrechts and S.~Schr{\"o}er.
\newblock {The Brauer group of analytic K3 surfaces}.
\newblock {\em Internat. Math. Res. Not.}, 2003(50):2687--2698, 2003.

\bibitem{kemeny-thesis}
M.~Kemeny.
\newblock {The universal Severi variety of rational curves on K3 surfaces}.
\newblock {\em Bull. London Math. Soc}, 45(1):159--174, 2012.

\bibitem{keum}
J.~Keum.
\newblock {A note on elliptic K3 surfaces}.
\newblock {\em Trans. AMS}, 352(5):2077--2086, 2000.

\bibitem{kleiman-special}
S.~L. Kleiman.
\newblock {$r$-Special subschemes and an argument of Severi's}.
\newblock {\em Adv. in Math}, 22:1--23, 1976.

\bibitem{knudsen}
F.~Knudsen.
\newblock {The projectivity of the moduli space of stable curves II: The stacks
  $M_{g,n}$}.
\newblock {\em Math. Scand.}, 52:161--199, 1983.

\bibitem{knut}
A.L. Knutsen.
\newblock {On kth-order embeddings of K3 surfaces and Enriques surfaces}.
\newblock {\em Manuscripta Math.}, 104:211--237, 2001.

\bibitem{kollar-simultaneous}
J.~Koll{\'a}r.
\newblock Simultaneous normalization and algebra husks.
\newblock {\em Asian J. Math.}, 15(3):437--450, 2011.

\bibitem{kool-thomas}
M.~Kool and R.P. Thomas.
\newblock {Reduced classes and curve counting on surfaces I: Theory}.
\newblock {\em Algebraic Geometry}, 1(3):334--383, 2014.

\bibitem{lazarsfeld-bnp}
R.~Lazarsfeld.
\newblock {Brill--Noether--Petri without degenerations}.
\newblock {\em J. Diff. Geom.}, 23:299--307, 1986.

\bibitem{liu}
Q.~Liu.
\newblock {\em {Algebraic Geometry and Arithmetic Curves}}.
\newblock Oxford University Press, United Kingdom, 2002.

\bibitem{mori-mukai}
S.~Mori and S.~Mukai.
\newblock The uniruledness of the moduli space of curves of genus 11.
\newblock In {\em Algebraic Geometry}, volume 1016 of {\em Lecture Notes in
  Math.}, pages 334--353. Springer Berlin, 1983.

\bibitem{morrison-large}
D.~Morrison.
\newblock {On K3 surfaces with large Picard number}.
\newblock {\em Invent. math.}, 75(1):105--121, 1984.

\bibitem{mukai-symplectic}
S.~Mukai.
\newblock {Symplectic structure of the moduli space of sheaves on an abelian or
  K3 surface}.
\newblock {\em Invent. Math.}, 77(1):101--116, 1984.

\bibitem{mukai-moduli}
S.~Mukai.
\newblock {On the moduli space of bundles on K3 surfaces, I}.
\newblock {\em Vector bundles on algebraic varieties (Bombay, 1984)},
  11:341--413, 1987.

\bibitem{mukai-fano}
S.~Mukai.
\newblock Fano 3--folds.
\newblock {\em Complex Projective Geometry: Selected Papers}, (179):255, 1992.

\bibitem{mukai-genus11}
S.~Mukai.
\newblock {Curves and K3 surfaces of genus eleven}.
\newblock In {\em Moduli of vector bundles ({S}anda, 1994; {K}yoto, 1994)},
  volume 179 of {\em Lecture Notes in Pure and Appl. Math.}, pages 189--197.
  Dekker, New York, 1996.

\bibitem{mukai-duality}
S.~Mukai.
\newblock {Duality of polarized K3 surfaces}.
\newblock In {\em New trends in algebraic geometry ({W}arwick, 1996)}, volume
  264 of {\em London Math. Soc. Lecture Note Ser.}, pages 311--326. Cambridge
  Univ. Press, Cambridge, 1999.

\bibitem{mukai-nonabelian}
S.~Mukai.
\newblock Non-abelian {B}rill--{N}oether theory and {F}ano 3-folds.
\newblock {\em Sugaku Expositions}, 14(2):125--153, 2001.
\newblock Sugaku Expositions.

\bibitem{mumf-chow}
D.~Mumford.
\newblock {Rational equivalence of $0$-cycles on surfaces}.
\newblock {\em J. Math. Kyoto Univ.}, 84(3):445--461, 1986.

\bibitem{ogrady}
K.G. O'Grady.
\newblock {Moduli of sheaves and the Chow group of K3 surfaces}.
\newblock {\em Journal de Math{\'e}matiques Pures et Appliqu{\'e}es},
  100(5):701--718, 2013.

\bibitem{donat}
B.~Saint-Donat.
\newblock {Projective models of K3 surfaces}.
\newblock {\em Amer. J. Math.}, 96:602--639, 1974.

\bibitem{schuett}
M.~Sch\"{u}tt and T.~Shioda.
\newblock {Elliptic surfaces}.
\newblock {\em Algebraic geometry in East Asia - Seoul 2008, Advanced Studies
  in Pure Mathematics}, 60:51--160, 2010.

\bibitem{sernesi-def}
E.~Sernesi.
\newblock {\em {Deformations of algebraic schemes}}.
\newblock Springer, Netherlands, 2006.

\bibitem{shimada-arxiv}
I.~Shimada.
\newblock {On elliptic K3 surfaces}.
\newblock arXiv:math/0505140.

\bibitem{shimada-mich}
I.~Shimada.
\newblock {On elliptic K3 surfaces}.
\newblock {\em Michigan Math. J}, 47(3):423--446, 2000.

\bibitem{teissier}
B.~Teissier.
\newblock R{\'e}solution simultan{\'e}e : I - familles de courbes.
\newblock {\em S{\'e}minaire sur les singularit{\'e}s des surfaces}, pages
  1--10, 1976-1977.

\bibitem{wahl-jac}
J.~Wahl.
\newblock {The Jacobian algebra of a graded Gorenstein singularity}.
\newblock {\em Duke Math. J.}, 55(4):843--511, 1987.

\bibitem{wahl-curve}
J.~Wahl.
\newblock Introduction to {G}aussian maps on an algebraic curve.
\newblock In {\em Complex projective geometry ({T}rieste, 1989/{B}ergen,
  1989)}, volume 179 of {\em London Math. Soc. Lecture Note Ser.}, pages
  304--323. Cambridge Univ. Press, Cambridge, 1992.

\bibitem{wahl-square}
J.~Wahl.
\newblock On the cohomology of the square of an ideal sheaf.
\newblock {\em J. Alg. Geom.}, 76:481--871, 1997.

\end{thebibliography}
\end{document}